\documentclass[12pt,a4paper,pdfps]{article}
\usepackage[cp1251]{inputenc}
\usepackage[english]{babel}
\usepackage{color}
\usepackage[colorlinks,linkcolor=blue,filecolor=blue,citecolor=blue]{hyperref}
\usepackage{amsmath,amssymb,amsthm}
\usepackage{graphicx}
\usepackage{comment}
\usepackage{enumitem}
\usepackage{subfigure,multirow}
\usepackage{amsthm}
\usepackage{thmtools}
\declaretheoremstyle[notefont=\bfseries,notebraces={}{},%
   headpunct={},postheadspace=1em,bodyfont=\it]{mystyle}
\declaretheorem[style=mystyle,numbered=no,name=Theorem]{thm-hand}
\newtheorem{thmletter}{Theorem}

\usepackage{geometry}
\DeclareMathOperator{\dist}{dist}
\DeclareMathOperator{\sgn}{sgn}
\DeclareMathOperator{\ttt}{type}

\DeclareMathOperator{\Ker}{Ker}
\DeclareMathOperator{\Kil}{Kil}
\DeclareMathOperator{\diag}{diag}
\DeclareMathOperator{\Lip}{Lip}
\DeclareMathOperator{\Image}{Im}
\DeclareMathOperator{\Real}{Re}

\DeclareMathOperator{\interior}{int}

\DeclareMathOperator{\ad}{\mathrm{ad}}
\DeclareMathOperator{\Ad}{\mathrm{Ad}}
\DeclareMathOperator{\Exp}{\mathrm{Exp}}
\DeclareMathOperator{\sspan}{\mathrm{span}}

\DeclareMathOperator{\atan}{\mathrm{arctan}}

\newcommand{\tmax}{t_{\mathrm{max}}}
\newcommand{\tconj}{t_{\mathrm{conj}}}
\newcommand{\tcut}{t_{\mathrm{cut}}}
\newcommand{\Rinj}{R_{inj}}
\newcommand{\ssl}{\mathfrak{sl}}

\newcommand{\g}{\mathfrak{g}}
\newcommand{\h}{\mathfrak{h}}

\newcommand{\m}{\mathfrak{m}}

\newcommand{\R}{\mathbb{R}}
\newcommand{\C}{\mathbb{C}}
\newcommand{\N}{\mathbb{N}}
\newcommand{\Z}{\mathbb{Z}}
\newcommand{\HH}{\mathbb{H}}

\newcommand{\id}{\mathrm{id}}
\newcommand{\SU}{\mathrm{SU}}
\newcommand{\U}{\mathrm{U}}
\newcommand{\AdS}{\mathrm{AdS}}
\newcommand{\SL}{\mathrm{SL}}

\newcommand{\SO}{\mathrm{SO}}

\newcommand{\const}{\mathrm{const}}

\newcommand{\hone}{\bar{h}_1}

\newcommand{\st}{\sin{\tau}}
\newcommand{\ct}{\cos{\tau}}
\newcommand{\sht}{\sinh{\tau}}
\newcommand{\cht}{\cosh{\tau}}
\newcommand{\se}{\sin{(\tau \eta \hone)}}
\newcommand{\ce}{\cos{(\tau \eta \hone)}}

\newcommand{\argl}{\frac{t \eta h_1}{2 I_2}}

\newcommand{\stl}{\sin{\argl}}
\newcommand{\ctl}{\cos{\argl}}

\newcommand{\typeh}{\ttt{(h)}}

\newcommand{\A}{\mathcal{A}}
\newcommand{\M}{\mathcal{M}}
\newcommand{\V}{\mathcal{V}}

\newcommand{\Conj}{\mathrm{Conj}}
\newcommand{\Cut}{\mathrm{Cut}}

\DeclareMathOperator{\arcsinh}{arcsinh}

\theoremstyle{definition}
\newtheorem{definition}{Definition}
\newtheorem{remark}{Remark}
\newtheorem{example}{Example}

\theoremstyle{plain}

\newtheorem{lemma}{Lemma}
\newtheorem{theorem}{Theorem}
\newtheorem{proposition}{Proposition}

\newenvironment{enumerate*}%
  {\begin{enumerate}%
    \setlength{\itemsep}{1pt}%
    \setlength{\parskip}{1pt}}%
  {\end{enumerate}}

\title{Cut loci of Berger type Lorentzian structures}

\author{
A.\,V.~Podobryaev \\ A.\,K.~Ailamazyan Program Systems
Institute of RAS \\ Pereslavl-Zalesskiy, Russia \\ \tt{alex@alex.botik.ru} 
}

\date{}

\begin{document}

\maketitle

\begin{abstract}
Consider the deformation of the standard Lorentzian metric on the anti de-Sitter space along the fibers of the Hopf fibration.
We study the universal covering of this Lorentzian manifold to exclude a priori presence of time-like cycles.
We describe the sets attainable by admissible curves and study the question of the existence of the longest arcs.
Next, we investigate Lorentzian geodesics for optimality: we find the cut time and the cut locus.
As a geometric application we compute the injectivity radius of the corresponding Lorentzian manifold.


\end{abstract}


\section*{\label{sec-introduction}Introduction}

Let us define the Lorentzian manifold we deal with. Consider the surface
$$
\HH^{1,n} = \{(z,w_1,\dots,w_n) \in \C^{n+1} \, | \, -|z|^2 + |w|^2 = -1\},
$$
where $w = (w_1,\dots,w_n) \in \C^n$ and $|w|^2 = |w_1|^2 + \dots + |w_n|^2$.
This manifold equipped with the restriction from $\C^{n+1}$ to $\HH^{1,n}$ of the quadratic form $-|z|^2 + |w|^2$ of the signature $(2,2n)$
is called \emph{the anti de-Sitter space} and is denoted by $\AdS_{2n+1}$. Notice, that the restricted quadratic form has the signature $(1,2n)$.

Next, consider the diagonal action of the circle, i.e., the Lie group $\U_1$,
$$
e^{i\varphi} \cdot (z,w) = (e^{i\varphi}z, e^{i\varphi}w_1, \dots, e^{i\varphi}w_n), \qquad (z,w) \in \HH^{1,n}, \qquad e^{i\varphi} \in \U_1.
$$
We get \emph{the Hopf fibration} $\HH^{1,n} \rightarrow \HH^{1,n} / \U_1 \subset \C P^n$.
In this paper, we consider Lorentzian quadratic forms on the manifold $\HH^{1,n}$, i.e., nondegenerate quadratic forms with signature $(1,2n)$, that are deformations of the anti de-Sitter Lorentzian form along the fibers of the Hopf fibration.
More precisely, fix the point $o = (1,0) \in \HH^{1,n}$, consider the tangent space $T_o\HH^{1,n} = \{(ia,w) \, | \, a \in \R, \, w \in \C^n \}$ and the quadratic form
$Q_o(a,w) = -I_1 a^2 + I_2 |w|^2$ on this space, where $I_1, I_2 > 0$ are fixed. Next, we define the quadratic form $Q$ on the whole space $\HH^{1,n}$ by the shifts of the form $Q_o$ via the transitive action of the Lie group $\U_{1,n}$ on the manifold $\HH^{1,n}$. Note, that this definition is correct, since the form $Q_o$ is invariant under the stabilizer of the point $o$ that is equal to
$$
\left\{
\left(
\begin{array}{ll}
1 & 0\\
0 & A\\
\end{array}
\right) \, \Bigm| \, A \in \U_n \right\} \simeq \U_n.
$$
Thus, we consider invariant Lorentzian quadratic forms on the homogeneous space $\U_{1,n}/\U_n$. We call them \emph{Berger type Lorentzian structures} by analogue with well known \emph{Berger spheres} $\U_{n+1} / \U_n \simeq S^{2n+1}$ with Riemannian metrics deformed along the fibers $S^1$ of the Hopf fibration $S^{2n+1} \rightarrow \C P^n$
given by the $\U_1$-action on the unit sphere $S^{2n+1} \subset \C^{n+1}$. Cut loci for Berger spheres were found by C.~Rakotoniaina~\cite{rakotoniaina}.

The considered Lorentzian structures are more complex than anti de-Sitter one. The classical anti de-Sitter space appears as a particular case when $I_1 = I_2$ and the corresponding homogeneous space $\AdS_{2n+1} = \SO_{2,2n}/\SO_{1,2n}$ has a richer symmetry group.
For a general Berger type Lorentzian structure the sectional curvature is not constant in contrast with the anti de-Sitter space.
Recently the interest in such spaces is rising~\cite{calvaruso-helix}.

Having a Lorentzian quadratic form $Q$ on a manifold $M$ one can define admissible curves and their lengths. A curve $\gamma: [0,t_1] \rightarrow M$ is called \emph{admissible} if almost all of its tangent vectors are \emph{light-like} or \emph{time-like} with respect to the Lorentzian quadratic form $Q$, i.e.,
$Q_{\gamma(t)}(\dot{\gamma}(t)) = 0$ or $Q_{\gamma(t)}(\dot{\gamma}(t)) < 0$ for almost all $t \in [0,t_1]$ and any velocity vector $\dot{\gamma}(t)$ must be inside or on the border of the same component of the light-cone as the fixed time-like vector field (called \emph{a time direction}).
\emph{The length of the curve} $\gamma$ is defined as
$\int\limits_0^{t_1}{\sqrt{|Q_{\gamma(t)}(\dot{\gamma}(t))|} \, dt}$.
\emph{The Lorentzian distance} from the point $q_0 \in M$ to the point $q_1 \in M$ is the supremum of Lorentzian lengths of all admissible curves from the point $q_0$ to the point $q_1$. If this supremum does not exist, then we say that the Lorentzian distance equals $+\infty$. If the curve where this supremum is achieved exists, then it is called \emph{the longest arc}.

It is easy to see that the existence of an admissible cycle implies the non-existence of the longest arc. Thus, it is natural to consider simply connected Lorentzian manifolds. So, we investigate the universal covering of the homogeneous space $\HH^{1,n} = \U_{1,n}/\U_n$ which we denote by $M = \widetilde{\HH}^{1,n}$.
The Lie group $\U_{1,n}$ is homotopic to its maximal compact subgroup $\U_1 \times \U_n$. Hence, for fundamental groups we have
$\pi_1(\U_n) = \Z \hookrightarrow \Z \times \Z = \pi_1(\U_1 \times \U_n) = \pi_1(\U_{1,n})$, this is the embedding to the second component.
It follows, that the universal covering $\widetilde{\U}_{1,n}$ acts transitively on the space $M = \widetilde{\U}_{1,n} / \widetilde{\U}_n$ and
the Lorentzian structure is invariant under this action.
Notice, that the universal covering of the anti de-Sitter space $\AdS_{2n+1}$ is also often called the anti de-Sitter space.

Our goal is to find the longest arcs.
We use the following coordinates on the manifold $M$ and the projection map:
$$
\Pi : M = \{(c,w) \, | \, c \in \R, w \in \C^n \} \rightarrow \HH^{1,n} = \U_{1,n}/\U_n, \qquad
\Pi(c,w) = \left( \sqrt{1+|w|^2}e^{ic}, w \right).
$$
Since our structure is $\widetilde{\U}_{1,n}$-invariant, it is sufficient to consider the fixed point $\tilde{o} = (0,0)$ of the manifold $M$ as a start point for admissible curves.

The basic case is the three dimensional case $\widetilde{\HH}^{1,1}$. Any case of bigger dimension $\widetilde{\HH}^{1,n}$ can be reduced to the three dimensional case with the help of the $\U_n$-symmetry.

We use methods of geometric control theory~\cite{agrachev-sachkov}. We find extremal trajectories called \emph{geodesics} with the help of Pontryagin maximum principle. It turns out that geodesics are orbits of one-parametric subgroups of the group of isometries. Then we find the attainable set $\A$ and study the question of the existence of the longest arcs on the attainable set. This depends on the ratio of the eigenvalues $-I_1$ and $I_2$.
We refer to the case $I_1 > I_2$ as \emph{the oblate case}, the case $I_1 = I_2$ as \emph{the symmetric case}, and the case $I_1 < I_2$ as \emph{the prolate case}.

\begin{thmletter}
\label{th-A}
We have the following description of the attainable set $\A$.\\
\emph{(1)} In the oblate case we can attain any point of the manifold $M$, i.e., $\A = M$.
The longest arcs do not exist.\\
\emph{(2)} In the symmetric case
$$
\A = \{(c,w) \in M \, | \, \arctan{|w|} \leqslant c\}.
$$
The longest arcs exist only for the terminal points located in the set
$$
\A_{exist} = \{(c,w) \in M \, | \, \arctan{|w|} \leqslant c < \pi - \arctan{|w|} \} \cup \{(\pi,0)\}.
$$
\emph{(3)} In the prolate case
$$
\A = \Biggl\{(c,w) \in M \, \Bigm| \, c \geqslant
\atan{\Biggl(\frac{\tan{(\tau\sqrt{\eta})} + \frac{1}{\sqrt{\eta}}\tanh{\tau}}{1-\frac{1}{\sqrt{\eta}}\tanh{\tau}\tan{(\tau\sqrt{\eta})}}\Biggr)}, \ \text{where} \
\tau = \arcsinh{\Bigl(|w|\frac{\sqrt{\eta}}{\sqrt{\eta+1}}\Bigr)} \Biggr\},
$$
where $\eta = \frac{I_2}{I_1} - 1$.
Any point of the set $\A$ can be reached from the point $\tilde{o}$ by the longest arc.
\end{thmletter}

We prove item~(1) of this theorem by constructing a sequence of constant controls to reach an arbitrary point.

In the symmetric case~(2) the set $\A_{extr}$ attainable by geodesics is well known.
It can be found in textbooks on the Lie theory, see a description of the Lie exponential map $\exp : \ssl_2(\R) \rightarrow \SL_2(\R)$, for example~\cite{hall}.
Then, one can prove that any point of the set $\A \supset \A_{extr}$ is attainable by admissible trajectories. Moreover, admissible trajectories can not leave the set $\A$,
since admissible velocities on its boundary are directed to its interior.
The attainable set in this case was described in paper~\cite{grong-vasiliev} studying the sub-Lorentzian structure on the manifold $\widetilde{\HH}^{1,1}$.
However, we put here this fact and its proof to illustrate the general situation.
But it seems to us that the geodesics optimality is not sufficiently investigated in this case.
Therefore, we explicitly indicate the set $\A_{exist}$ where the longest arcs exist.
It remains to mention a similar study of two-dimensional anti de-Sitter space~\cite{ali-sachkov} which is an axial section of our symmetric case.

Next, we prove item~(3) using the sufficient condition for the existence of the longest arcs on the attainable set developed in paper~\cite{lokutsievskiy-podobryaev}.
It turns out that the longest arcs exist on the attainable set in this case.
Hence, thanks to the Pontryagin maximum principle the set $\A_{extr}$ attainable by geodesics coincide with the whole attainable set $\A$.
Moreover, the set $\A_{extr}$ has a simple description, since we know the equations of geodesics.

For the symmetric and the prolate cases we find \emph{the cut locus} and \emph{the cut time}, i.e., the set of points and the time where geodesics lose their optimality.
Since in the oblate case the longest arcs do not exist, then the cut locus question is senseless in this case.

Let us use the following conventions. Let us denote by $\g$ and $\h$ the Lie algebras of the Lie groups $\U_{1,n}$ and $\U_n$, respectively.
Since the stabilizer $\U_n$ is compact we have a reductive decomposition of the Lie algebra $\g = \m \oplus \h$,
where $\m$ is $(\Ad{\U_n})$-invariant and the subspaces $\m$ and $\h$ are orthogonal with respect to the Killing form $\Kil$ on the Lie algebra $\g$,
see~\cite[Prop.~1]{kowalski-szenthe}. Moreover, the subspace $\m$ can be identified with the tangent space $T_{\tilde{o}}M$.

Any geodesic is a projection of a trajectory of the Hamiltonian vector field on the cotangent bundle $T^*M$.
The initial data of such a trajectory is a covector from the space $T^*_{\tilde{o}}M \simeq \m^*$ which is canonically isomorphic to the annihilator $\h^{\circ} \subset \g^*$.
Thus, the cut time is a function $\tcut: \h^{\circ} \rightarrow \R_+ \cup \{+\infty\}$.
Since the Killing form delivers the isomorphism $\g \simeq \g^*$,
we can consider it as a non-degenerate quadratic form on the space $\g^*$ which we will also denote by $\Kil$.

\begin{thmletter}
\label{th-B}
The cut locus and the cut time have the following description, where $\Kil$ is the Killing form and $|h| = \sqrt{|\Kil(h)|}$ for $h \in \g^*$.\\
\emph{(1)} For the symmetric case $I_1 = I_2$
$$
\Cut = \{(\pi, 0)\}, \qquad
\tcut(h) = \left\{
\begin{array}{lll}
\frac{2 \pi I_2}{|h|}, & \text{if} & \Kil{h} < 0,\\
+\infty, & \text{if} & \Kil{h} = 0,\\
\end{array}
\right. \qquad h \in \h^{\circ},
$$
there are no geodesics with initial covectors $h$ such that $\Kil{h} > 0$.\\
\emph{(2)} For the prolate case $I_1 < I_2$
$$
\Cut = \left\{(c,0) \, \Bigm| \, c \geqslant \frac{\pi I_2}{I_1} \right\}, \qquad
\tcut(h) = \left\{
\begin{array}{lll}
\frac{2 \pi I_2}{|h|}, & \text{if} & \Kil{h} < 0,\\
+\infty, & \text{if} & \Kil{h} \geqslant 0,\\
\end{array}
\right. \qquad h \in \h^{\circ}.
$$
\end{thmletter}

We use the symmetry method developed in~\cite{sachkov-didona} to prove this theorem.
The main idea is to study the diffeomorphic properties of the exponential map for the corresponding optimal control problem and to find maximal domains in its pre-image and image diffeomorphic to each other.

This result allows us to compute the injectivity radius of our Lorentzian metric.
Recall that the injectivity radius is the supremum of the ball radii in $\m^*$ (with respect to some reference Riemannian metric $g_R$) such that the exponential map restricted to these balls is a diffeomorphism.

\begin{thmletter}
\label{th-C}
The injectivity radius is equal to
$$
\Rinj = \left\{
\begin{array}{lll}
0, & \text{if} & I_1 > I_2,\\
2\pi \frac{I_2}{\sqrt{|\lambda|}}, & \text{if} & I_1 \leqslant I_2,\\
\end{array}
\right.
$$
where $\lambda$ is the negative eigenvalue of the Killing form with respect to the Riemannian metric $g_R$.
\end{thmletter}

Let us say a few words about some related studies.
A.\,Z.~Ali and Yu.\,L.~Sachkov~\cite{ali-sachkov} recently considered the 2-dimensional anti de-Sitter space $\mathrm{AdS}_2 = \SO^+_{1,2}/\SO_{1,1}$.
For the natural Lorentzian structure on this space they found the attainable set, studied the optimality of geodesics, described the Lorentzian distance and spheres.
The Lorentzian manifold $\AdS_2$ is an axial section of the Lorentzian manifold considered in this paper in a particular symmetric case.
Also, we would like to mention here the research of D.-Ch.~Chang, E.~Grong, I.~Markina, A.~Vasil'ev~\cite{chang-markina-vasilev,grong-vasiliev} dedicated to the sub-Lorentzian structure on the manifold $\widetilde{\HH}^{1,1}$. In our recent work~\cite{podobryaev-lorentz-sl2-so11-sym}, we study a one-parametric family
of Lorentzian structures on this manifold, these structures are $\SO_{1,1}$-symmetric and the limit case of this one-parametric family is the sub-Lorentzian case.
In other words, in paper~\cite{podobryaev-lorentz-sl2-so11-sym}, we consider invariant Lorentzian structures on the homogeneous space
$\left( \SU_{1,1} \times \SO_{1,1} \right) / \SO_{1,1}$, where the stabilizer is embedded in the direct product in the anti-diagonal way.

This paper is organized as follows. We deal with the three dimensional case $\widetilde{\HH}^{1,1}$ in Sections~\ref{sec-problem-statement}--\ref{sec-inj-rad} and prove there Theorems~\ref{th-A}--\ref{th-B} for this case. Then, in Section~\ref{sec-general}, we reduce the general case to the three dimensional one. Let us describe the context of the sections dedicated to the three dimensional case more precisely. We give the statement of the optimal control problem in Section~\ref{sec-problem-statement}, also we discuss the place of our problem among general Lorentzian structures on the manifold $\widetilde{\HH}^{1,1}$ and compute the sectional curvature of the corresponding Lorentzian manifold. In Section~\ref{sec-parametrization} we get the parametrization of Lorentzian geodesics via Pontryagin's maximum principle. Section~\ref{sec-atset} is dedicated to the attainable sets and the question of the existence of the longest arcs for different cases, Theorem~\ref{th-A} is almost proved there except the part of item~(2) about the existence of the longest arcs to the points of the set $\A_{exist}$. Next, in Section~\ref{sec-def} we put some general definitions which are necessary for the study of the cut locus in the next three sections. We find the conjugate time in Section~\ref{sec-conj-time}, the Maxwell time for rotation symmetries of our problem in Section~\ref{sec-maxwell-time}, and finally we describe the cut locus and the cut time in Section~\ref{sec-cc-cut-locus} and prove Theorem~\ref{th-B}.
After that we return to the remaining part of Theorem~\ref{th-A} and complete its proof.
As an application of the main result on the cut locus we compute the injectivity radius for the considered Lorentzian metric (Theorem~\ref{th-C}), see Section~\ref{sec-inj-rad}.

\section{\label{sec-problem-statement}Optimal control problem in the three dimensional case}

In this section, we discuss the place of the considered Lorentzian structure among possible invariant Lorentzian structures on the manifold $\widetilde{\HH}^{1,1}$. After that we state the optimal control problem for the Lorentzian longest arcs.

First, note that the homogeneous space
$$
\HH^{1,1} = \U_{1,1}/\U_1 = \U_{1,1} / \{\diag{(1,e^{i\varphi})} \, | \, \varphi \in \R \}
$$
is equivalent to the homogeneous space
$$
(\SU_{1,1} \times \U_1) / \U_1 = \left\{\left(A, \diag{(e^{i\varphi},e^{-i\varphi})}\right) \, | \, A \in \SU_{1,1}, \, \varphi \in \R \right\} /
\left\{\left(\diag{(e^{-i\varphi},e^{i\varphi})}, \diag{(e^{i\varphi},e^{-i\varphi})}\right) \, | \, \varphi \in \R \right\},
$$
where the stabilizer $\U_1$ is embedded in the direct product $\SU_{1,1} \times \U_1$ in the anti-diagonal way.
Indeed, there is the two-leaves covering
$$
\Psi : \SU_{1,1} \times \U_1 \rightarrow \U_{1,1}, \qquad
\Psi \left( A, \diag{(e^{i\varphi},e^{-i\varphi})} \right) = A \cdot \diag{(e^{i\varphi},e^{i\varphi})},
$$
$$
\Psi \left( \diag{(e^{-i\varphi},e^{i\varphi})}, \diag{(e^{i\varphi},e^{-i\varphi})} \right) = \diag{(1,e^{2i\varphi})},
$$
so the anti-diagonal stabilizer $\U_1$ covers the stabilizer $\U_1 = \{ \diag{(1,e^{i\varphi})} \, | \, \varphi \in \R \}$ twice.
Notice, that the homogeneous space $(\SU_{1,1} \times \U_1) / \U_1$ is the famous Selberg's example of a weakly symmetric homogeneous space~\cite{selberg}.

In other words, we study a left-invariant Lorentzian structure on the universal covering of the Lie group $\SU_{1,1} \simeq \SL_2(\R)$ that is also $\U_1$-invariant, i.e., has the circle symmetry.
We denote this group by $G = \widetilde{\SU}_{1,1}$. Let us look at a general left-invariant Lorentzian structure on the Lie group $G$.

A left-invariant Lorentzian structure on a three dimensional Lie group $G$ is defined by a nondegenerate quadratic form $Q$ of the signature $(1,2)$ on the corresponding Lie algebra $\g$.
What are the parameters of this quadratic form? It is natural to diagonalize the quadratic form with respect to the Killing form $\Kil$.
But since the Killing form is also indefinite in our case, then it is not always possible. The famous Finsler's lemma gives the sufficient condition.

\begin{lemma}[\cite{finsler}]
\label{lem-finsler}
Let $R$ and $Q$ be quadratic forms on a finite dimensional vector space $V$. Assume that for any nonzero $v \in V$ such that $Q(v) = 0$ it follows that $R(v) > 0$.
Then there exists a number $\lambda$ such that $R + \lambda Q$ is positive definite.
\end{lemma}

If we assume that the quadratic form $Q$ satisfies Finsler's condition of Lemma~\ref{lem-finsler} with respect to the Killing form $\Kil$ or vice versa, then the following proposition implies that such a quadratic form is defined by three parameters.

\begin{proposition}
\label{prop-canonical-form}
Let $Q$ be a quadratic form on $\g$ with signature $(1, 2)$.
Assume that the Killing form $\Kil(v)$ has the same sign for any $v \in \g$ such that $Q(v) = 0$.
Then there is a basis $e_1, e_2, e_3$ of the Lie algebra $\g$ which is isomorphic to $\ssl_2(\R)$ such that
\begin{equation}
\label{eq-commutators}
[e_1, e_2] = e_3, \qquad [e_1, e_3] = -e_2, \qquad [e_2, e_3] = -e_1
\end{equation}
and the form $Q$ has the matrix $\diag{(-I_1, I_2, I_3)}$ in this basis, where $I_1, I_2, I_3 > 0$.
\end{proposition}

\begin{proof}
There exists a number $\lambda$ such that the quadratic form $S = \pm\Kil + \lambda Q$ is positive definite due to Finsler's Lemma~\ref{lem-finsler}.
This implies that there exists a basis $f_1,f_2,f_3$ of the Lie algebra $\g$ such that the quadratic forms $S$ and $Q$ have the matrixes
$\diag{(1,1,1)}$ and $\diag{(a,b,c)}$ in this basis, respectively.
Hence, the matrix of the form $\pm\Kil$ is equal to $\diag{(1-\lambda a, 1 - \lambda b, 1 - \lambda c)}$.
Since the Killing form has signature $(1,2)$, we may assume that $1-\lambda a < 0$, $1 - \lambda b > 0$, $1 - \lambda c > 0$ without loss of generality.
Then put $e_1 = \frac{f_1\sqrt{2}}{\sqrt{\lambda a - 1}}$, $e_2 = \frac{f_2\sqrt{2}}{\sqrt{1 - \lambda b}}$, $e_3 = \frac{f_3\sqrt{2}}{\sqrt{1 - \lambda c}}$.
In the new basis $e_1,e_2,e_3$ the Killing form reads as $\diag{(-2,2,2)}$, while the form $Q$ is still diagonal.
Due to Finsler's condition the signs of its diagonal elements coincide with the signs of the elements of the Killing form.
So, the $Q$-matrix has the form $\diag{(-I_1,I_2,I_3)}$ for some $I_1,I_2,I_3 > 0$.
It remains to note that in the following basis of the Lie algebra $\g$:
$$
\bar{e}_1 = \frac{1}{2}\left(
\begin{array}{rr}
i & 0\\
0 & -i\\
\end{array}
\right), \qquad
\bar{e}_2 = \frac{1}{2}\left(
\begin{array}{rr}
0 & 1\\
1 & 0\\
\end{array}
\right), \qquad
\bar{e}_3 = \frac{1}{2}\left(
\begin{array}{rr}
0 & i\\
-i & 0\\
\end{array}
\right),
$$
the Killing form has the matrix $\diag{(-2,2,2)}$.
Moreover, the commutator relations of these elements coincide with~\eqref{eq-commutators}.
So, the linear transformation from the basis $\bar{e}_1,\bar{e}_2,\bar{e}_3$ to the basis $e_1,e_2,e_3$ belongs to the Lie group $\SO_{1,2}$.
Whence, this transformation is a result of the adjoint action, in particular, it is an automorphism of the Lie algebra $\g$.
It follows that the elements $e_1,e_2,e_3$ have the same commutators.
\end{proof}

The left-shifts of the form $Q$ give a left-invariant Lorentzian structure on the group $G$. We consider the corresponding Lorentzian problem as an optimal control problem where the goal is to find an admissible curve connecting the two given points that maximize the Lorentzian length computed with the help of the quadratic form Q.
We will call a curve $g : [0,t_1] \rightarrow G$ \emph{admissible} if for almost all $t \in [0,t_1]$ for its tangent vector we have $Q(L_{g(t) *}^{-1}\dot{g}(t))) \leqslant 0$,
where $L_g$ is a left-shift by an element $g \in G$. In other words, admissible velocities are \emph{not spacelike}.
Furthermore, an admissible curve must be future directed, i.e.,
the vector $\dot{g}(t)$ and the vector of time direction $L_{g(t) *} e_1$ are located in the same component of the non-spacelike cone.
Of course, due to left-invariance of the problem we may assume that the starting point coincides with the identity element $\id \in G$.
So, the problem reads as
\begin{equation}
\label{eq-optimal-control-problem}
\begin{array}{ll}
g \in \Lip{([0, t_1], G)}, & u \in L^{\infty}([0, t_1], U), \\
g(0) = \id, \qquad g(t_1) = g_1, & \dot{g}(t) = L_{g(t) *} (u_1(t) e_1 + u_2(t) e_2 + u_3(t) e_3), \\
\int\limits_0^{t_1}{\sqrt{I_1 u_1^2(t) - I_2 u_2^2(t) - I_3 u_3^2(t)} \, dt} \rightarrow \max, & \\
\end{array}\underline{}
\end{equation}
where the terminal time $t_1$ is free and $U = \{u = (u_1,u_2,u_3) \in \R^3 \, | \, I_1 u_1^2 \geqslant I_2 u_2^2 + I_3 u_3^2, \ u_1 > 0\}$ is the control set.

The Lorentzian manifold in the symmetric case $I_1 = I_2 = I_3$ is the universal covering of the well known anti-de\,Sitter space that is a homogeneous space
$$
\AdS_3 = \SO_{2,2} / \SO_{1,2} = \SU_{1,1} \times \SU_{1,1} / \SU_{1,1}.
$$
The second equality here is due to the isomorphism $\SO^+_{2,2} \simeq \SU_{1,1} \times \SU_{1,1}$, see for example survey~\cite[Sec.~3.1]{bonsante-seppi}.
The corresponding Lorentzian quadratic form is just the Killing form. Indeed, the two copies of $\SU_{1,1}$ act via left- and right-shifts, the stabilizer is the diagonal subgroup of the direct product and the Killing form is bi-invariant.
The anti de-Sitter space $\AdS_3$ is a three dimensional Lorentzian manifold of constant negative sectional curvature.

In this paper, we consider a more complicate case that still has rotation symmetry, i.e., the axisymmetric case $I_2 = I_3$.
The next proposition characterises the sectional curvature in our case, in particular, it is not constant.

Let us introduce a parameter $\eta = \frac{I_2}{I_1} - 1$ which measures the prolateness of the Lorentzian quadratic form in the axisymmetric case.
Below we denote the bilinear form associated with the quadratic form $Q$ by the same letter.

\begin{proposition}
Assume that $I_2 = I_3$.
Then the sectional curvature as a function of a two-dimensional plane $P_w = \Ker{Q(w,\,\cdot\,)}  \subset T_{\id}G$ where $w \in T_{\id}G$ reads as
$$
K(P_w) = -\frac{1}{4I_1(\eta+1)^2} \left( 1 - \frac{4\eta Q(w,e_1)^2}{I_1 Q(w,w)} \right).
$$
\end{proposition}

\begin{proof}
Since the quadratic form $Q$ is axisymmetric, then without loss of generality we may assume that $w = ae_1 + be_3$.
Hence, the plane $P_w$ is spanned by two vectors:
$$
u = (\eta+1)be_1 + ae_3, \qquad v = e_2.
$$
So, we can compute $K(P_w) = \frac{Q(R(u,v)u, v)}{Q(u,u)Q(v,v) - Q(u,v)^2}$, using the definition of the Riemann curvature tensor
$$
R(u,v)u = \nabla_u \nabla_v u - \nabla_v \nabla_u u - \nabla_{[u,v]} u,
$$
the Koszul formula that in the left-invariant case reads as
$$
Q(\nabla_{e_i} e_j, e_k) = {{1}\over{2}} \Bigl(
Q([e_i,e_j],e_k) - Q([e_j,e_k],e_i) + Q([e_k, e_i],e_j)
\Bigr)
$$
and the commutator relations~\eqref{eq-commutators}.
\end{proof}

\begin{remark}
\label{rem-plots}
We plot graphs of geodesics, attainable set etc. in the $(c,w)$-space, where $c \in \R$ and $w \in \C$ are coordinates in the group $G$ which is the universal covering of $\SU_{1,1}$. Note that we consider the $c$-axis as the vertical axis. This is done so that the cone of admissible velocities is directed upward. We plot all the graphs in the plane $\Image{w} = 0$. So, the attainable sets are figures of revolution.
\end{remark}

\section{\label{sec-parametrization}Lorentzian geodesics}

In this section, we write the Hamiltonian system for geodesics, it is convenient to do it in the general case.
Then we derive the parametric equations for geodesics in the axisymmetric case.

Introduce the following family of functions on the cotangent bundle $T^*G$ depending on the parameters $u = (u_1, u_2, u_3) \in \R^3$ and $\nu \geqslant 0$:
$$
H_u^{\nu}(\lambda) = u_1 h_1(\lambda) + u_2 h_2(\lambda) + u_3 h_3(\lambda) + \nu \sqrt{I_1 u_1^2 - I_2 u_2^2 - I_3 u_3^2}, \qquad \lambda \in T^*G,
$$
where $h_i(\lambda) = \langle L_{\pi(\lambda) *} e_i, \lambda \rangle$, $i=1,2,3$  and
$\pi : T^*G \rightarrow G$ is the canonical projection.
So, here $L_{\pi(\lambda) *} e_i$ are the left-invariant vector fields (corresponding to the elements $e_i$) regarded as linear functions on the fibers of the cotangent bundle $T^*G$.

\begin{theorem}[Pontryagin's maximum principle~\cite{pontryagin,agrachev-sachkov}]
\label{th-pmp}
Let a pair of a trajectory $\hat{q} \in \Lip{([0, t_1], G)}$ and a control $\hat{u} \in L^{\infty}([0, t_1], U)$ is an optimal process for the problem~\emph{\eqref{eq-optimal-control-problem}}, then there exist a curve $\lambda \in \Lip{([0, t_1], T^*G)}$ and
a number $\nu \geqslant 0$ such that $(\lambda, \nu) \neq 0$, $\pi \circ \lambda = \hat{q}$ and \\
\emph{(1)} $\dot{\lambda}(t) = \vec{H}^{\nu}_{\hat{u}(t)} (\lambda(t))$ for a.e. $t \in [0, t_1]$,\\
\emph{(2)} $H^{\nu}_{\hat{u}(t)}(\lambda(t)) = \max\limits_{u \in U}{H^{\nu}_u(\lambda(t))}$,\\
\emph{(3)} $H^{\nu}_{\hat{u}(t)}(\lambda(t)) = 0$,\\
where $\vec{H}$ is the Hamiltonian vector field corresponding to the Hamiltonian $H$
with respect to the standard symplectic structure on the cotangent bundle $T^*G$.
\end{theorem}

Recall that a curve $\lambda$ is called \emph{a normal extremal} if $\nu \neq 0$ and \emph{an abnormal extremal} if $\nu = 0$.
The projection $\pi(\lambda)$ is called \emph{a normal} / \emph{an abnormal extremal trajectory} (\emph{geodesic}).
If there is no normal extremal that projects to an abnormal geodesic, then this geodesic is called \emph{a strict abnormal geodesic}.

Any extremal is defined by its initial condition that is a covector from $\g^*$.
This covector is called \emph{an initial covector} (\emph{an initial momentum}) for corresponding geodesic.

\begin{definition}
\label{def-types-of-geodesics}
\emph{A time-like geodesic} is a normal geodesic such that its velocities are located in the left-shifts of the interior $\interior{U}$ of the control set.
\emph{A light-like geodesic} is a normal geodesic such that its velocities are in the left-shifts of the set $U \setminus \interior{U}$.
\end{definition}

Instead of Lorentzian length one can consider the cost functional that quadratically depends on a control.
Such functional is much more practical since it does not contain a square root.
It turns out that extremal trajectories for the quadratic cost functional geometrically coincide with Lorentzian geodesics.

\begin{proposition}
\label{prop-energy}
\emph{(1)} The only abnormal geodesics for problem~\emph{\eqref{eq-optimal-control-problem}} are light-like, these geodesics are strict abnormal.\\
\emph{(2)} Any normal geodesic is time-like.\\
\emph{(3)} Geodesics geometrically coincide with images of the trajectories of the following Hamiltonian system under the projection $\pi$\emph{:}
$$
\dot{h}_i = \{H, h_i\}, \ i=1,2,3, \qquad \dot{g}(t) = L_{g(t) *} d_hH,
$$
where the functions $h_i$ and
\begin{equation}
\label{eq-H}
H = -\frac{1}{2}\left(\frac{h_1^2}{I_1} - \frac{h_2^2}{I_2} - \frac{h_3^2}{I_3}\right)
\end{equation}
are seen as functions on the dual space of the Lie algebra $\g^*$ and $\{ \,\cdot\, , \,\cdot\, \}$ is the standard Poisson structure on the space $\g^*$.\\
\emph{(4)} A geodesic is time-like \emph{(}respectively, light-like\emph{)} iff its initial covector lies on the level surface $H = -\frac{1}{2}$, $h_1 < 0$
\emph{(}respectively, $H = 0$, $h_1 < 0$\emph{)}.
\end{proposition}

\begin{proof}
This follows from the results of paper~\cite{podobryaev-extr}.
For~(1) see Corollary~2, item~(2) follows from Theorem~1, and items (3) and (4) from Corollary~1 and Remark~8.
\end{proof}

The subsystem for $h_i$ of the Hamiltonian system (see Proposition~\ref{prop-energy}~(3)) is called \emph{the vertical part of the Hamiltonian system}.

\begin{proposition}
\label{prop-vert-hamiltonian-system}
The vertical part of the Hamiltonian system reads as
\begin{equation*}
\label{eq-vertical-subsystem}
\dot{h}_1 = \frac{I_2 - I_3}{I_2 I_3} h_2 h_3, \qquad
\dot{h}_2 = \frac{I_1 - I_3}{I_1 I_3} h_1 h_3, \qquad
\dot{h}_3 = \frac{I_2 - I_1}{I_1 I_2} h_1 h_2.
\end{equation*}
\end{proposition}

\begin{proof}
It is a direct computation of Poisson brackets from Proposition~\ref{prop-energy}~(3) using $\{h_i, h_j\} = \langle [e_i, e_j], \,\cdot\, \rangle$, $i,j=1,2,3$ and the commutator relations~\eqref{eq-commutators}.
\end{proof}

Consider now the case $I_2 = I_3$ more precisely. Recall that the parameter $\eta$ measures the oblateness of the set of vectors with unit norm
$$
\eta = \frac{I_2}{I_1} - 1 > -1.
$$
We will refer to the cases $\eta < 0$, $\eta = 0$ and $\eta > 0$ as \emph{the oblate case}, \emph{the symmetric case} and \emph{the prolate case}, respectively.

It follows from Proposition~\ref{prop-vert-hamiltonian-system} that for $I_2 = I_3$ we get the following vertical subsystem:
$$
\begin{array}{l}
\dot{h}_1 = 0, \\
\dot{h}_2 = -\frac{\eta}{I_2} h_1 h_3, \\
\dot{h}_3 = \frac{\eta}{I_2} h_1 h_2. \\
\end{array}
$$
Introduce the following notation
$$
|h| = \sqrt{|\Kil(h)|}, \qquad \bar{h}_i = \frac{h_i}{|h|}, \ i=1,2,3, \qquad
\tau = \frac{t|h|}{2I_2}.
$$
Using this notation we can write the solutions of the vertical subsystem:
$$
\begin{array}{ccl}
\bar{h}_1 & = & \const, \\
\left(
\begin{array}{l}
\bar{h}_2(\tau) \\
\bar{h}_3(\tau) \\
\end{array}
\right)
& = & R_{2\tau\eta\bar{h}_1}
\left(
\begin{array}{l}
\bar{h}_2(0) \\
\bar{h}_3(0) \\
\end{array}
\right), \\
\end{array}
$$
where we denote by $R_{\alpha}$ a rotation around the point $(0,0)$ in the plane $(h_2,h_3)$ by an angle $\alpha$.

\begin{proposition}
\label{prop-extremal-trajectories}
A geodesic with an initial covector $h \in \g^*$ has the following parametrization\emph{:}
$$
g(t) = \exp{\frac{t}{I_2}\bigl(-h_1e_1 + h_2e_2 + h_3e_3\bigr)} \cdot \exp{\left(-t\frac{\eta h_1}{I_2} e_1\right)}.
$$
\end{proposition}

\begin{proof}
Let us check that $g(t)$ is a solution of the horizontal part of the Hamiltonian system by direct computation.
As usual we denote by $L_g$ and $R_g$ left- and right-shifts by an element $g \in G$, respectively.
To simplify these computations we use the notation $Z = \frac{\eta h_1}{I_2} e_1$.
$$
\dot{g}(t) = L_{\exp{\frac{t}{I_2}\left(-h_1e_1 + h_2e_2 + h_3e_3\right)} *} R_{\exp{(-tZ)} *} \frac{1}{I_2}\bigl(-h_1e_1 + h_2e_2 + h_3e_3\bigr) +
L_{\exp{\frac{t}{I_2}\left(-h_1e_1 + h_2e_2 + h_3e_3\right)} *} L_{\exp{(-tZ)} *} (-Z) =
$$
$$
= L_{g(t) *} \Bigl( \bigl(\Ad{\exp{(tZ)}}\bigr) \frac{1}{I_2}\bigl(-h_1e_1 + h_2e_2 + h_3e_3\bigr) - Z \Bigr).
$$
Substituting $Z = \frac{\eta h_1}{I_2} e_1$ we obtain
$$
\Bigl(\Ad{\exp{(tZ)}}\Bigr) \frac{1}{I_2}\Bigl(-h_1e_1 + h_2e_2 + h_3e_3\Bigr) - Z  = \frac{1}{I_2}\Bigl(-h_1(t)e_1 + h_2(t)e_2 + h_3(t)e_3\Bigr) - \frac{1}{I_2}\Bigl(\frac{I_2}{I_1}-1\Bigr)h_1e_1 =
$$
$$
= -\frac{h_1(t)}{I_1}e_1 + \frac{h_2(t)}{I_2}e_2 + \frac{h_3(t)}{I_3}e_3 = u_1e_1 + u_2e_2 + u_3e_3.
$$
So, we get $\dot{g}(t) = L_{g(t) *} (u_1e_1+u_2e_2+u_3e_3)$.
\end{proof}

\begin{remark}
\label{rem-isometry-group}
The parametrization of geodesics from Proposition~\ref{prop-extremal-trajectories} as a product of two one-parametric subgroups is not occasional.
This Proposition~\ref{prop-extremal-trajectories} follows from some general facts.
Let us mention these general considerations in addition to the explicit proof of this Proposition.
Note that the group $\R$ acts on $\g$ by rotations in the plane $\sspan{\{e_2, e_3\}}$ with the fixed vector $e_1$. Moreover, these transformations are automorphisms of the Lie algebra $\g$. It follows that the dual action on $\g^*$ induces the action of the same group on the group $G$ by isometries, see~\cite[Th.~1]{podobryaev-symmetries}.
This means that the isometry group of the Lorentzian structure on the group $G$ contains the group $G \times \R$.
The corresponding isotropy subgroup is
$$
K = \left\{ \left((a, 0), -a\right) \, | \, a \in \R\right\}.
$$
Since the vertical part of the Hamiltonian vector field is tangent to the orbits of the isotropy group $K$,
then the geodesic corresponding to the initial covector $h \in \g^*$ is an orbit of one-parametric subgroup of isometries
$$
g(t) = \exp{t\left(d_hH + Z\right)},
$$
where $Z$ is an element of the Lie algebra of the isotropy subgroup such that $\dot{h}(0) = -(\ad^* Z)h(0)$, see~\cite[Lemma~3.4]{podobryaev-hg}.
Note that in that paper this fact is proved for sub-Riemannian structure but the proof for arbitrary Hamiltonian system is quite the same.

We have the following elements of the Lie algebra $\g \oplus \R$:
$$
d_hH = \left( -\frac{h_1}{I_1}e_1 + \frac{h_2}{I_2}e_2 + \frac{h_3}{I_2}e_3, 0 \right), \qquad
Z = \left( \frac{\eta h_1}{I_2} e_1, -\frac{\eta h_1}{I_2} e_1 \right).
$$
Finally, computing the exponent for the sum of commuting elements of the Lie algebra, we get the product of exponents as in Proposition~\ref{prop-extremal-trajectories}.
\end{remark}

Now let us deduce geodesic equations in coordinates.
These equations can be three different types depending on an initial covector $h$ type with respect to the Killing form $\Kil{h} = -h_1^2 + h_2^2 + h_3^2$.

\begin{proposition}
\label{prop-circle-geodesics-in-coordinates}
A normal geodesic with an initial covector $h \in \g^*$ has the following parametrization\emph{:}
$$
c(\tau) = \arg{\{(q_0(\tau) + iq_1(\tau)\}}, \qquad w(\tau) = q_2(\tau) + iq_3(\tau),
$$
where $q_0, q_1, q_2, q_3$ are defined as follows.
We assume that the number $c(\tau) \in [0, +\infty)$ counts the turns of the curve $q_0(t) + iq_1(t) \in \C$ around zero.\\
\emph{(1)} If $\Kil{h} < 0$, then
\begin{equation}
\label{eq-circle-time-geodesic}
\begin{array}{ccl}
q_0(\tau) & = & \ct \ce - \hone \st \se,\\
q_1(\tau) & = & -\ct \se - \hone \st \ce,\\
\left(
  \begin{array}{l}
     q_2(\tau)\\
     q_3(\tau)\\
  \end{array}
\right) & = &  \st R_{\tau \eta \hone}
\left(
  \begin{array}{l}
     \bar{h}_2\\
     \bar{h}_3\\
  \end{array}
\right).\\
\end{array}
\end{equation}
\emph{(2)} If $\Kil{h} = 0$, then
\begin{equation}
\label{eq-circle-light-geodesic}
\begin{array}{ccl}
q_0(t) & = & \ctl - \frac{t}{2 I_2} h_1 \stl,\\
q_1(t) & = & -\stl - \frac{t}{2 I_2} h_1 \ctl,\\
\left(
  \begin{array}{l}
     q_2(t)\\
     q_3(t)\\
  \end{array}
\right) & = & \frac{t}{2 I_2} R_{\argl}
\left(
  \begin{array}{l}
     h_2\\
     h_3\\
  \end{array}
\right).\\
\end{array}
\end{equation}
\emph{(3)} If $\Kil{h} > 0$, then
\begin{equation}
\label{eq-circle-space-geodesic}
\begin{array}{ccl}
q_0(\tau) & = & \cht \ce - \hone \sht \se,\\
q_1(\tau) & = & -\cht \se - \hone \sht \ce,\\
\left(
  \begin{array}{l}
     q_2(\tau )\\
     q_3(\tau )\\
  \end{array}
\right) & = & \sht R_{\tau \eta \hone}
\left(
  \begin{array}{l}
     \bar{h}_2\\
     \bar{h}_3\\
  \end{array}
\right).\\
\end{array}
\end{equation}
\end{proposition}

\begin{proof}
The proofs are direct computations of the product of two exponents from Proposition~\ref{prop-extremal-trajectories}, using the
following formulas of the exponential map $\exp : \mathfrak{su}_{1,1} \rightarrow \SU_{1,1}$ for $x \in \mathfrak{su}_{1,1}$
$$
\exp(x) =
\left(
\begin{array}{cc}
q_0+iq_1 & q_2+iq_3\\
q_2-iq_3 & q_0 - iq_1\\
\end{array}
\right) \in \SU_{1,1},
$$
where
\begin{equation}
\label{eq-group-exp}
\begin{array}{llll}
q_0 + i q_1 = \cos{(\frac{|x|}{2})} + i \bar{x}_1 \sin{(\frac{|x|}{2})}, & q_2 + i q_3 = \sin{(\frac{|x|}{2})} (\bar{x}_2  + i \bar{x}_3), & \text{if} & \Kil(x) < 0,\\
q_0 + i q_1 = 1 + i \frac{x_1}{2}, & q_2 + i q_3 = \frac{1}{2}(x_2  + i x_3), & \text{if} & \Kil(x) = 0,\\
q_0 + i q_1 = \cosh{(\frac{|x|}{2})} + i \bar{x}_1 \sinh{(\frac{|x|}{2})}, & q_2 + i q_3 =  \sinh{(\frac{|x|}{2})} (\bar{x}_2 + i \bar{x}_3), & \text{if} & \Kil(x) > 0,\\
\end{array}
\end{equation}
and $x = x_1e_1 + x_2e_2 + x_3e_3 \in \g$.

Let us do it for the case $\Kil(h) < 0$, in the other cases computations are quite the same.
By formula~\eqref{eq-group-exp}, remembering that $\tau = \frac{t|h|}{2I_2}$ and since $\Kil(-h_1e_1 + h_2e_2 +h_3e_3)$ is also negative we have
$$
\begin{array}{rl}
\exp{\frac{t}{I_2}(-h_1e_1 + h_2e_2 +h_3e_3)} =
& \left(
\begin{array}{cc}
\cos{\tau} - i\bar{h}_1\sin{\tau} & \sin{(\tau)}(\bar{h}_2+i\bar{h}_3)\\
\sin{(\tau)}(\bar{h}_2-i\bar{h}_3) & \cos{\tau} - i\bar{h}_1\sin{\tau}\\
\end{array}
\right),\\
\exp{(-\frac{t\eta h_1}{I_2}e_1)} =
& \left(
\begin{array}{cc}
\cos{(\tau\eta\bar{h}_1)} - i\sin{(\tau\eta\bar{h}_1)} & 0\\
0 & \cos{(\tau\eta\bar{h}_1)} + i\sin{(\tau\eta\bar{h}_1)}\\
\end{array}
\right).
\end{array}
$$
It remains to calculate the matrix product in the group $\SU_{1,1}$.
\end{proof}

\section{\label{sec-atset}Attainable sets and existence of the solution}

We begin with few lemmas which are common to three cases (oblate, symmetric and prolate) and then we consider each case separately.

First of all we need the multiplication rule for elements of the group $G$. It is not difficult to deduce it from the definition, see also~\cite[Def.~1]{grong-vasiliev}.
Let $(c_1, w_1) \cdot (c_2, w_2) = (c, w)$, then
\begin{equation}
\label{eq-mult}
\begin{array}{c}
c = c_1 + c_2 + \atan{\frac{\Image{(w_1\bar{w}_2e^{-i(c_1+c_2)})}}{\sqrt{1+|w_1|^2}\sqrt{1+|w_2|^2} + \Real{(w_1\bar{w}_2e^{-i(c_1+c_2)})}}},\\
w = w_2\sqrt{1+|w_1|^2}e^{ic_1} + w_1\sqrt{1+|w_2|^2}e^{-ic_2}.\\
\end{array}
\end{equation}
Let us denote by $\A$ the attainable set from the point $\id$.

\begin{lemma}
\label{lem-admissible-velocities}
Any admissible velocity at a point $(c_0, w_0) \in G$ has the form\emph{:}
$$
\left(
\xi + \frac{\Image{w_0\bar{\omega}e^{-ic_0}}}{\sqrt{1+|w_0|^2}}, \
\omega\sqrt{1+|w_0|^2}e^{ic_0} - iw_0\xi
\right),
$$
where $\xi \in \R$ and $\omega \in \C$ are such that $\xi \geqslant \sqrt{\eta+1}|\omega| > 0$.
\end{lemma}

\begin{proof}
Let $(\xi, \omega) \in \g$ be an admissible velocity at the identity point.
Since $I_1|\xi|^2 - I_2 |\omega|^2 \geqslant 0$ and $(\xi, \omega) \neq (0, 0)$ we obtain $\xi \geqslant \sqrt{\eta+1}|\omega| > 0$.

Consider a curve $(c_1(t), w_1(t)) \in G$, $t \in [0, 1]$ such that
$$
c_1(0) = 0, \qquad \dot{c}_1(0) = \xi, \qquad w_1(0) = 0, \qquad \dot{w}_1(0) = \omega.
$$
It remains to compute the velocity of the left-shift of this curve by an element $(c_0, w_0)$, i.e.,
$$
\frac{d}{dt}\Bigm|_{t=0} (c_0, w_0) \cdot (c_1(t), w_1(t)).
$$
This is a direct computation using the multiplication law~\eqref{eq-mult}.
\end{proof}

\begin{lemma}
\label{lem-round}
If $(c_0,w_0) \in \A$, then $\{(c_0,w) \in G \, | \, |w| = |w_0|\} \subset \A$.
\end{lemma}

\begin{proof}
Assume that the point $(c_0,w_0)$ is attainable by a curve $\gamma : [0,t_1] \rightarrow G$ with the help of a control $u : [0,t_1] \rightarrow U$.
This means that $\dot{\gamma}(t) = L_{\gamma(t) *} u(t)$.
Consider an automorphism $r_{\varphi} = \ad{(\varphi e_1)}$ of the Lie algebra $\g$ for an arbitrary $\varphi$.
The corresponding automorphism $R_{\varphi}$ of the Lie group $G$ is the rotation by the angle $\varphi$ with respect to the $c$-axis.
Since $R_{\varphi}$ is an automorphism we obtain
$R_{\varphi} \dot{\gamma}(t) = L_{R_{\varphi} \gamma(t) *} r_{\varphi} u(t)$.
It follows, that the curve $R_{\varphi}\gamma(\cdot)$ is an admissible curve with the control $r_{\varphi} u(\cdot)$
connecting the points $\id$ and $(c,R_{\varphi}w_0)$.
\end{proof}

\begin{lemma}
\label{lem-att-up}
If $(c_0,w_0) \in \A$, then $\{(c,w_0) \in G \, | \, c \geqslant c_0\} \subset \A$.
\end{lemma}

\begin{proof}
Consider the admissible curve with the constant control $u = (1,0,0)$ starting from the point $\id$.
Using formula~\eqref{eq-group-exp} we obtain the parametrization of this curve $\gamma_u(t) = (t/2, 0)$.
Starting from the point $(c_0,w_0)$ with this constant control $u$, with the help of the multiplication rule~\eqref{eq-mult}, we get the curve
$(c_0, w_0) \cdot (t/2, 0) = (c + t/2, w_0 e^{-t/2})$.
It is sufficient to apply Lemma~\ref{lem-round} to finish the proof.
\end{proof}

\subsection{\label{sec-atset-oblate}The oblate case}

First consider the oblate case $\eta \in (-1,0)$.
We prove that in this case the longest arcs do not exist.
More precisely, we show the complete controllability using sequences of constant controls.
So, as a consequence we prove that there is a closed admissible curve passing through any point.
This implies non existence of the longest arcs.

\begin{lemma}
\label{lem-pi-4}
If $\eta \in (-1,0)$, then the following inclusion is satisfied\emph{:}
$\left\{(c,w) \in G \, | \, c \geqslant \atan{\frac{\sqrt{\eta+1}}{\sqrt{-\eta}}} \right\} \subset \A$.
\end{lemma}

\begin{proof}
Consider the admissible curve $\gamma(\cdot)$ with the constant light-like control $u = (\frac{\sqrt{\eta+1}}{\sqrt{-\eta}},\frac{1}{\sqrt{-\eta}},0)$.
Note that $\Kil{u} = 1 > 0$.
According to formulas~\eqref{eq-group-exp} we obtain
$$
\gamma(t) = \left(\arg{\left\{\cosh{\frac{t}{2}} + i\frac{\sqrt{\eta+1}}{\sqrt{-\eta}} \sinh{\frac{t}{2}}\right\}}, \frac{1}{\sqrt{-\eta}}\sinh{\frac{t}{2}}\right).
$$
Note that $\arg{\{\cosh{\frac{t}{2}} + i\frac{\sqrt{\eta+1}}{\sqrt{-\eta}}\sinh{\frac{t}{2}}\}} \rightarrow \atan{\frac{\sqrt{\eta+1}}{\sqrt{-\eta}}}-0$
while $t \rightarrow +\infty$.
So, the required statement follows from Lemmas~\ref{lem-round}--\ref{lem-att-up}.
\end{proof}

\begin{lemma}
\label{lem-step}
Let $\eta \in (-1,0)$.
Assume that for fixed $c \in \R$ and $r > 0$ the set $\A_{c,r} = \{(c, w) \in G \, | \, |w| > r\}$ contains in the attainable set $\A$.
Then there exist $\varepsilon > 0$ and $r_1 > 0$ such that $\A_{(c-\varepsilon), r_1} \subset \A$.
\end{lemma}

\begin{proof}
Consider an admissible curve starting from the identity point with a constant light-like control $u$ such that
$\bar{u} = \left(\frac{\sqrt{\eta+1}}{\sqrt{-\eta}},\bar{u}_2,\bar{u}_3\right)$ where $\sqrt{\bar{u}_2^2 + \bar{u}_3^2} = \frac{1}{\sqrt{-\eta}}$.
By formulas~\eqref{eq-group-exp} since $\Kil{u} = 1 > 0$ this curve have the parametrization
\begin{equation}
\label{eq-light}
c_u(t) = \atan{\left(\frac{\sqrt{\eta+1}}{\sqrt{-\eta}}\tanh{\frac{t}{2}}\right)}, \qquad w_u(t) = \sinh{\left(\frac{t}{2}\right)} (\bar{u}_2+i\bar{u}_3).
\end{equation}
Let us consider the family of such curves starting from the point $(c_0, w_0) \in \A_{c,r}$ such that $w_0 \in \R$, $w_0 > 0$
(it is possible due to Lemma~\ref{lem-round}).
We get the family of curves $(\hat{c}_u(t), \hat{w}_u(t)) = (c_0, w_0) \cdot (c_u(t), w_u(t))$.

Consider the following function $f: [0, +\infty) \rightarrow \R$
$$
f(s) = \frac{s}{\sqrt{1+s^2}}.
$$
It is easy to see that this function increase up to $1$. Indeed,
\begin{equation}
\label{eq-lim}
f'(s) = \frac{1}{(1+s^2)^{3/2}} > 0, \qquad \lim\limits_{s \rightarrow +\infty}{\frac{s}{\sqrt{1+s^2}}} = 1.
\end{equation}
Moreover, $\lim\limits_{s \rightarrow 0+}{\frac{f(s)}{f(as)}} = \frac{1}{a}$ for $a \neq 0$.

In particular, in the case $\eta < 0$ we get
$$
\lim\limits_{s \rightarrow 0+}{\frac{\sqrt{\eta+1}}{\sqrt{-\eta}} \frac{f\left(\sinh{\frac{t}{2}}\right)}{f(|w_u(t)|)}} = \sqrt{\eta+1} < 1.
$$
Indeed, here $s = \sinh{(\frac{t}{2})}$ and from~\eqref{eq-light} we have $|w_u(t)| = \frac{s}{\sqrt{-\eta}}$ and $a = \frac{1}{\sqrt{-\eta}}$.

Thus, one can choose some fixed $t_1$ such that
$$
\frac{\sqrt{\eta+1}}{\sqrt{-\eta}} \frac{f\left(\sinh{\frac{t_1}{2}}\right)}{f(|w_u(t_1)|)} < 1.
$$
Due to~\eqref{eq-lim} there exists $r_0 > r > 0$ such that for any $|w_0| > r_0$ we obtain
\begin{equation}
\label{eq-w0}
\frac{\sqrt{\eta+1}}{\sqrt{-\eta}} \frac{f\left(\sinh{\frac{t_1}{2}}\right)}{f(|w_u(t_1)|)} < f(r_0) < f(|w_0|) < 1.
\end{equation}
It follows that
$$
c_u(t_1) = \atan{\left(\frac{\sqrt{\eta+1}}{\sqrt{-\eta}}\tanh{\frac{t_1}{2}}\right)}  =
\atan{\left(\frac{\sqrt{\eta+1}}{\sqrt{-\eta}} f\left(\sinh{\frac{t_1}{2}}\right)\right)} <
$$
$$
< \atan{\left(f(r_0)f(|w_u(t_1)|)\right)} < \atan{\left(f(|w_0|)f(|w_u(t_1)|)\right)},
$$
for any $|w_0| > r_0$ since the functions $\atan$ and $f$ increase.

We can choose $u$ such that $\arg{\{w_0\bar{w}_u(t_1)e^{-i(c_0 + c_u(t_1))}\}} = -\frac{\pi}{2}$. Using the multiplication rule~\eqref{eq-mult} we can obtain
\begin{equation}
\label{eq-new-c}
\hat{c}_u(t_1) = c_0 + c_u(t_1) - \atan{\frac{|w_0|}{\sqrt{1 + |w_0|^2}} \frac{|w_u(t_1)|}{\sqrt{1 + |w_u(t_1)|^2}}}.
\end{equation}

From~\eqref{eq-new-c} we get
$$
\hat{c}_u(t_1) - c_0 = \atan{\left(\frac{\sqrt{\eta+1}}{\sqrt{-\eta}} f\left(\sinh{\frac{t_1}{2}}\right)\right)} - \atan{\left(f(|w_0|)f(|w_u(t_1)|)\right)} < 0.
$$
Put $\varepsilon = c_0 - \hat{c}_u(t_1) > 0$.
Next, using the multiplication rule~\eqref{eq-mult} note that
$$
\hat{w}_u(t_1) = w_u(t_1) e^{ic_0}\sqrt{1 + |w_0|^2} + w_0 e^{-ic_u(t_1)}\sqrt{1 + |w_u(t_1)|^2}.
$$
This is equal to
\begin{equation}
\label{eq-w}
\hat{w}_u(t_1) = e^{i(2c_0 + c_u(t_1))} \left(
w_u(t_1)e^{-i(c_0+c_u(t_1))}\sqrt{1 + |w_0|^2} + w_0 e^{-2i(c_0 + c_u(t_1))}\sqrt{1 + |w_u(t_1)|^2} \right).
\end{equation}
Let us compute the angle between two terms of the sum in the brackets.
It is equal to
$$
\arg{\{w_0 e^{-2i(c_0 + c_u(t_1))}\}} - \arg{\{w_u(t_1)e^{-i(c_0+c_u(t_1))}\}} = \arg{\{w_0 \bar{w}_u(t_1)e^{-i(c_0+c_u(t_1))}\}} = -\frac{\pi}{2}
$$
due to the our choice of $u$.
Thus, for the absolute value of $\hat{w}_u(t_1)$ by the Pythagoras theorem we have
$$
|\hat{w}_u(t_1)|^2 = |w_u(t_1)|^2 (1 + |w_0|^2) + |w_0|^2 (1 + |w_u(t_1)|^2).
$$
So, choosing suitable $|w_0| > r_0$ we can get any value of $|\hat{w}_u(t_1)|$ greater than
$$
r_1 = \sqrt{|w_u(t_1)|^2 (1 + r_0^2) + r_0^2 (1 + |w_u(t_1)|^2)}.
$$
Due to inequality~\eqref{eq-w0} we get that for the chosen $w_0$ the corresponding value $\hat{c}_u(t_1) < c_0 - \varepsilon$.
It is sufficient to apply Lemmas~\ref{lem-round}--\ref{lem-att-up} to conclude that $\A_{(c-\varepsilon), r_1} \subset \A$.
\end{proof}

\begin{lemma}
\label{lem-far}
Let $\eta \in (-1,0)$.
For any $c \in \R$ there exists $w \in \C$ \emph{(}for sufficiently large $|w|$\emph{)} such that $(c, w) \in \A$.
\end{lemma}

\begin{proof}
Indeed, starting from the set $\{(\atan{\frac{\sqrt{\eta+1}}{\sqrt{-\eta}}}, w) \, | \, w\in \C\} \subset \A$ (see Lemma~\ref{lem-pi-4}) and applying Lemma~\ref{lem-step} many times can
we get any $c \in \atan{\frac{\sqrt{\eta+1}}{\sqrt{-\eta}}} - \varepsilon \N$. Next, by Lemma~\ref{lem-att-up} we get any $c \in \R$.
\end{proof}

\begin{proposition}
\label{prop-full-control}
Assume that $\eta \in (-1, 0)$.\\
\emph{(1)} The system is full controlled, i.e., $\A = G$.\\
\emph{(2)} There exists an admissible cycle passing through every point.\\
\emph{(3)} There are no optimal solutions.
\end{proposition}

\begin{proof}
(1) For the given $c \in \R$ let us start from the point $(c_0,w_0) = (c - 2\atan{\frac{\sqrt{\eta+1}}{\sqrt{-\eta}}} - \frac{\pi}{2}, w_0) \in \A$ for sufficiently large $|w_0|$ by curves $(\hat{c}_u(t),\hat{w}_u(t)) = (c_0,w_0) \cdot (c_u(t), w_u(t))$ with constant light-like controls. The start point belongs to the attainable set by Lemma~\ref{lem-far}.
It follows from~\eqref{eq-w} that for suitable $w_0$ and $t_1$ we can get $\hat{w}_u(t_1) = 0$.
Indeed, due to~\eqref{eq-lim} we can choose sufficiently big $t_1$ such that
$$
\frac{|w_u(t_1)|}{\sqrt{1+|w_u(t_1)|^2}} = \frac{|w_0|}{\sqrt{1+|w_0|^2}}.
$$
Thus, we may assume the two complex numbers in brackets in formula~\ref{eq-w} have the same absolute value.
Next, choose $\arg{w_0}$ such that the arguments of these numbers become opposite.

Moreover, since for light-like control we have $c_u(t_1) < \atan{\frac{\sqrt{\eta+1}}{\sqrt{-\eta}}}$, see~\eqref{eq-light}.
Hence, from the multiplication rule~\eqref{eq-mult} it follows that $\hat{c}_u(t_1) < c - \atan{\frac{\sqrt{\eta+1}}{\sqrt{-\eta}}}$.
By Lemma~\ref{lem-att-up} we see that $(c - \atan{\frac{\sqrt{\eta+1}}{\sqrt{-\eta}}}, 0) \in \A$. Finally, applying Lemma~\ref{lem-pi-4} we get $\{(c, w) \, | \, w \in \C\} \subset \A$,
it follows that $\A = G$.

(2) Let us show that there exists an admissible cycle passing through the point $\id$. First, by~(1) the point $(-1, 0)$ is attainable.
Second, start from the point $(-1, 0)$ with control $(1, 0, 0)$ and get the point $\id$.

(3) immediately follows from~(2).
\end{proof}

So, item~(1) of Theorem~\ref{th-A} in the three dimensional case follows from Proposition~\ref{prop-full-control}.

\subsection{\label{sec-atset-symmetric}The symmetric case}

In this subsection, we assume that $\eta = 0$. We describe the attainable set along admissible trajectories and the attainable set along geodesics which are different in this case. The results of this section in the part concerning these attainable sets are well known, see~\cite[Prop.~6]{grong-vasiliev}.
See also the structure of the Lie exponential map for the group $\SL_2(\R) \simeq \SU_{1,1}$ in any textbook on Lie theory, for example~\cite[Ex.~3.22]{hall}.
However, we put here these results in order to illustrate the contrast with other cases.
It seems to us that optimality of geodesics have not been studied completely in this case.
We refer to paper~\cite{ali-sachkov} were the 2-dimensional anti de-Sitter space is considered which is the axial section of our Lorentzian manifold in the case $\eta = 0$.

Let us prove that the attainable set is $\A = \{(c,w) \in G \, | \, c \geqslant \atan{|w|}\}$ in the symmetric case.
To do this we consider light-like extremal trajectories from the point $(0,0)$ (see Lemma~\ref{lem-light-like-boundary} below) and
then we prove that the cone of admissible velocities at any point of a light-like extremal trajectory is directed inside the set $\A$ (see Lemma~\ref{lem-inside}).
After that we deduce that generally speaking the corresponding Lorentzian problem has no optimal solution.
The main argument is that the attainable set along extremal trajectories does not coincide with the attainable set $\A$ along admissible trajectories (see Lemma~\ref{lem-atset-extr} and Fig.~\ref{pic-symmetric-case-atset}).

\begin{lemma}
\label{lem-light-like-boundary}
The light-like trajectories starting from the point $(0,0)$ sweep the surface defined by equation $c = \atan{|w|}$.
The normal vector to this surface at a point $(c_0,w_0) \neq (0,0)$ is proportional to the vector $\left(1+|w_0|^2, -\frac{w_0}{|w_0|}\right)$.
\end{lemma}

\begin{proof}
It follows from formulas~\eqref{eq-circle-light-geodesic} that we have the following parametrization of a light-like extremal trajectory
$$
c(t) = \arg{\left\{1-i\frac{t}{2I_2}h_1\right\}}, \qquad w(t) = \frac{t}{2I_1}(h_2 + ih_3), \qquad h_1^2 = h_2^2 + h_3^2.
$$
This implies that
$$
c(t) = \atan{\left(\frac{t}{2I_2}|h_1|\right)}, \qquad |w(t)| = \frac{t}{2I_2}|h_1|.
$$
So, we get $c(t) = \atan{|w(t)|}$ for the light-like extremal trajectories from the point $(0,0)$.
Moreover, since $\atan'{|w|} = \frac{1}{1+|w|^2}$ we get the required normal vector to the corresponding surface of revolution.
\end{proof}

\begin{lemma}
\label{lem-inside}
The Euclidian scalar product of any admissible velocity at a point $(c_0, w_0) \neq (0,0)$ and the vector
$$
\left(1+|w_0|^2, -\frac{w_0}{|w_0|}\right)
$$
is non negative.
\end{lemma}

\begin{proof}
Let $(a_1, b_1), (a_2, b_2) \in \R \times \C$ be two vectors. Their scalar product reads as $a_1a_2 + \Real{b_1\bar{b}_2}$.
Substituting an admissible velocity at a point $(c_0, w_0)$ from Lemma~\ref{lem-admissible-velocities} and the vector from the statement of Lemma~\ref{lem-inside} we obtain
$$
(1+|w_0|^2)\xi + \sqrt{1+|w_0|^2}\Image{\left(w_0\bar{\omega}e^{-ic_0}\right)} - \Real{\left(\bar{w}_0\omega e^{ic_0}\right)}\frac{\sqrt{1+|w_0|^2}}{|w_0|} + \Real{i\frac{w_0\bar{w}_0}{|w_0|}\xi}.
$$
Note that the last term equals zero.
Without loss of generality we may assume that $\xi = 1$, it follows that $|\omega| = 1$ due to $\eta = 0$.
Denote $z = w_0\bar{\omega}e^{-ic_0}$, then $\bar{z} = \bar{w}_0\omega e^{ic_0}$.
Since $\omega$ is an arbitrary complex number on the unit circle, then $|z| = |w_0|$ and the sign of the upper expression equals to the sign of
$$
\sqrt{1+|z|^2} \left(\sqrt{1+|z|^2} + \Image{z} - \frac{\Real{\bar{z}}}{|z|} \right),
$$
where the first multiplier is positive.
Denoting $\arg{z}$ by $\alpha$ we get that the second multiplier is equal to
$$
\sqrt{1+|z|^2} + |z|\sin{\alpha} - \cos{\alpha} = \sqrt{1+|z|^2}\left(1 + \frac{|z|}{\sqrt{1+|z|^2}}\sin{\alpha} - \frac{1}{\sqrt{1+|z|^2}}\cos{\alpha}\right),
$$
where the first multiplier is positive, while the second is non negative, since it equals
$$
1 - \cos{(\alpha + \varphi)} \geqslant 0,
$$
where $\varphi = \arccos{\frac{1}{\sqrt{1+|z|^2}}}$.
\end{proof}

\begin{lemma}
\label{lem-atset-extr}
The attainable set along extremal trajectories reads as
$$
\A_{extr} = \{(c,w) \in G \, | \, c = \arctan{|w|}\} \cup \{(c,w) \in G \, | \, \pi k + \arctan{|w|} < c < \pi(k+1) - \arctan{|w|}, \ k = 0,1,\dots\}.
$$
\end{lemma}

\begin{proof}
See Fig.~\ref{pic-symmetric-case-atset}~(b).
We can attain the set $\{(c,w) \in G \, | \, c = \arctan{|w|}\}$ by the light-like extremal trajectories as it was shown in Lemma~\ref{lem-light-like-boundary}.
Consider now the time-like extremal trajectories. By formulas~\eqref{eq-circle-time-geodesic} we get the following parametrization of a time-like extremal trajectory:
$$
c(t) = \arg{\{\cos{\tau} - i\bar{h}_1\sin{\tau}\}}, \qquad w(t) = \sin{\tau} (\bar{h}_2 + i\bar{h}_3), \qquad \tau = \frac{t|h|}{2I_2}.
$$
Denoting by $c_1(t)$ the smallest distance from the point $c(t)$ to the points $\pi k$, where $k=0,1,\dots$, we obtain
$$
|\tan{c_1(t)}| = |\bar{h}_1| |\tan{\tau}|, \qquad |w(t)| = |\sin{\tau}| \sqrt{\bar{h}_2^2 + \bar{h}_3^2} = |\sin{\tau}| \sqrt{\bar{h}_1^2 - 1}.
$$
It follows that $\frac{|\tan{c_1(t)}|}{|w(t)|} = \frac{|\bar{h}_1|}{\sqrt{\bar{h}_1^2 - 1}}\frac{1}{|\cos{\tau}|}$.
Finally, notice that this expression gets any value in the interval $(1, +\infty)$ while $\bar{h}_1 \in [1, +\infty)$.
\end{proof}

\begin{figure}[t]
\minipage{0.32\textwidth}
  \includegraphics[width=\linewidth]{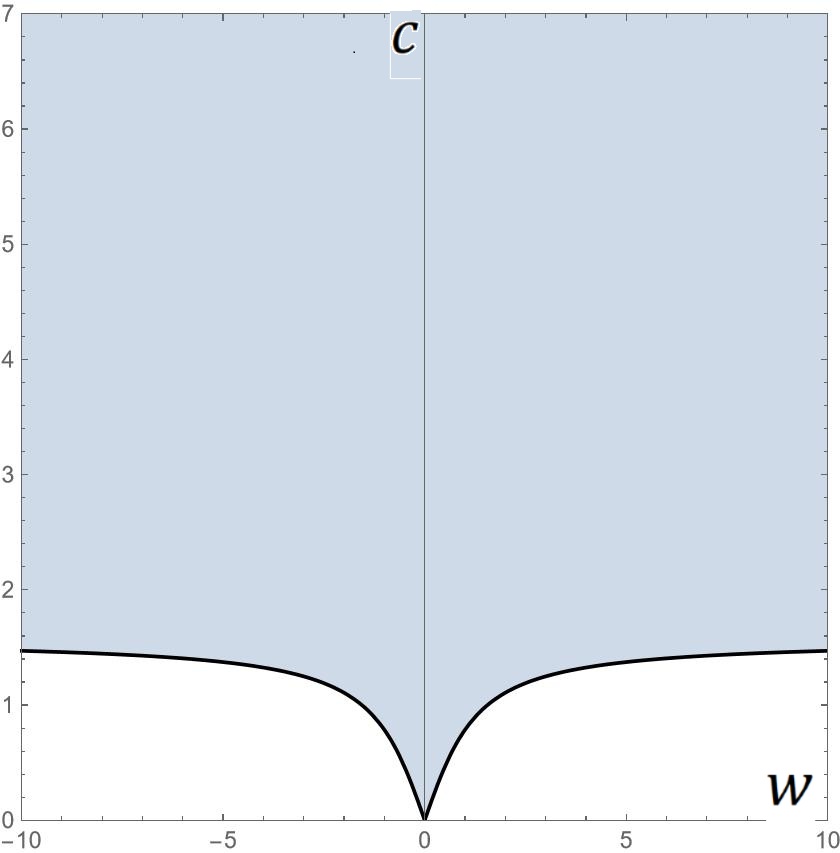}
  \\ \center{(a) The set attainable by admissible trajectories.}
\endminipage\hfill
\minipage{0.32\textwidth}
  \includegraphics[width=\linewidth]{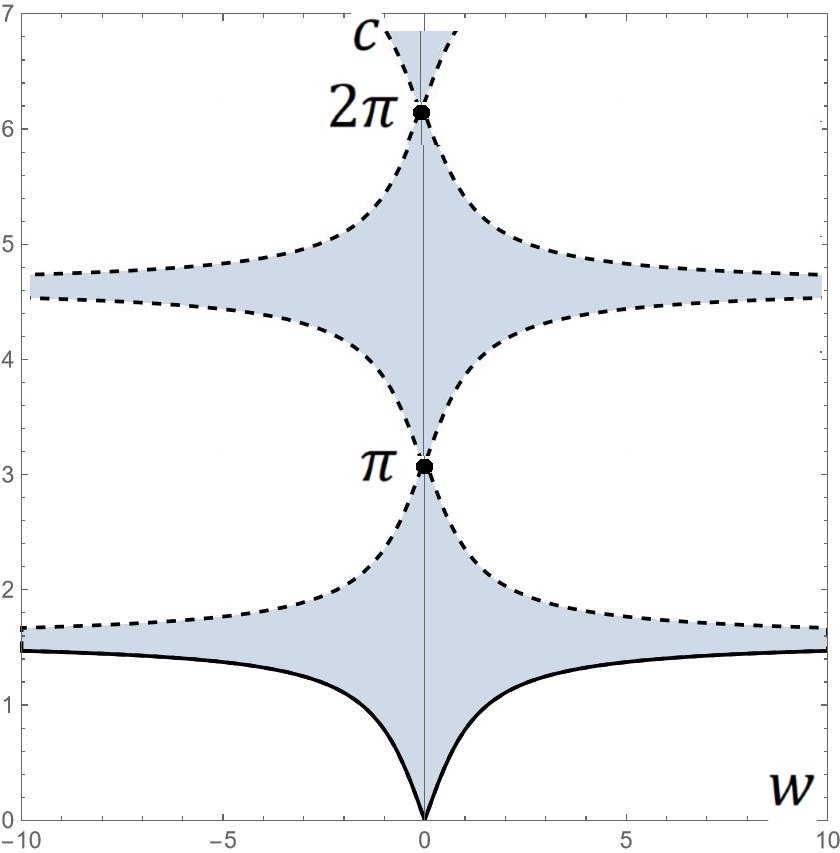}
  \\ \center{(b) The set attainable by geodesics.}
\endminipage\hfill
\minipage{0.32\textwidth}
  \includegraphics[width=\linewidth]{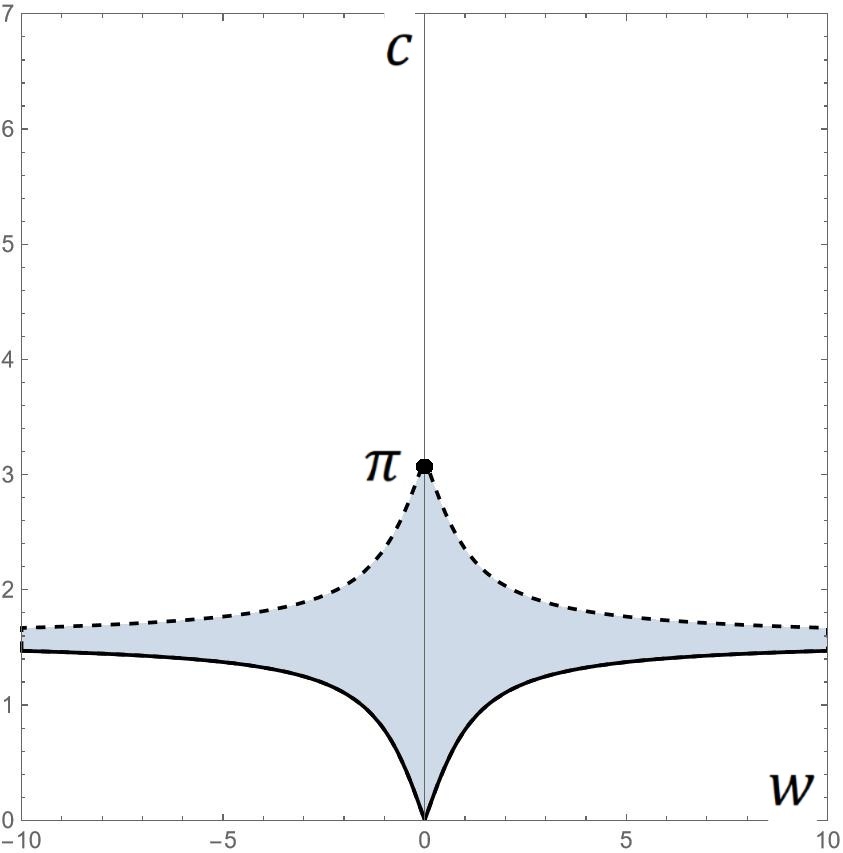}
  \\ \center{(c) The set where the longest arcs exist.}
\endminipage
\caption{\label{pic-symmetric-case-atset}The attainable sets for the symmetric case (the surface of revolution with respect to the $c$-axis, see also Remark~\ref{rem-plots}). The solid line is $c = \atan{|w|}$, the dashed lines are $c = \pi k \pm \atan{|w|}$, $k \in \N$. The dashed lines do not belong to the sets. Note that in spite of $c$ taking the first position in the pair $(c,w)$ the $c$-axis is vertical.}
\end{figure}

\begin{proposition}
\label{prop-symmetric-case-atset}
Assume that $\eta = 0$.\\
\emph{(1)} Attainable set from the identity point reads as
$$
\A = \{(c,w) \in G \, | \, c \geqslant \atan{|w|}\}.
$$
\emph{(2)} For any point of the set $\A \setminus \A_{extr} \neq \varnothing$ the longest arc coming from the identity point does not exist.\\
\emph{(3)} Moreover, there are no longest arcs coming to the terminal points from the set
$\{(c,w) \, | \, c > \pi - \atan{|w|}\}$.
\end{proposition}

\begin{proof}
(1) It follows from Lemmas~\ref{lem-light-like-boundary} and \ref{lem-att-up} that the set $\A$ contains in the attainable set.
Moreover, by Lemma~\ref{lem-inside} we can't leave the set $\A$. This implies that the set $\A$ coincides with the attainable set, see Fig.~\ref{pic-symmetric-case-atset}~(a).

(2) The set $\A \setminus \A_{extr}$ is not empty by Lemma~\ref{lem-atset-extr}, see Fig.~\ref{pic-symmetric-case-atset}~(a--b). We claim that for any point of this set there are no longest arcs from the identity point to this point. If by contradiction such an arc exists then by the Pontryagin maximum principle this arc is a part of some extremal trajectory.

(3) It is sufficient to show that the Lorentzian distance from the point $(0,0)$ to the point $(c,w)$ such that $c = \pi - \atan{|w|} + \varepsilon$ is infinite for any $\varepsilon > 0$. To do this we construct the family of admissible curves from the point $(0,0)$ to the point $(c,w)$ with arbitrary big Lorentzian lengths.
Namely, any such curve is a concatenation of three arcs. The first and the third arcs are light-like, while the second one is vertical with the Euclidian length $\varepsilon$,
see Fig.~\ref{pic-longarc}. It remains to show that this vertical arc is admissible and can have arbitrary big Lorentzian length, see Lemma~\ref{lem-symmetric-case-no-solution} below.
\end{proof}

\begin{figure}[h]
  \center{\includegraphics[width=0.32\linewidth]{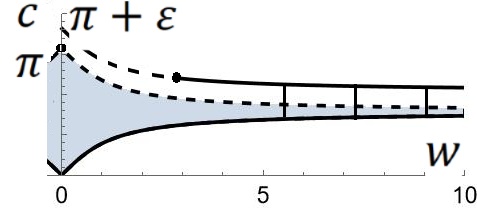}}
\caption{\label{pic-longarc}The sequence of arcs with the Lorentzian length tending to infinity.}
\end{figure}

\begin{lemma}
\label{lem-symmetric-case-no-solution}
Let $(c_0,w_0) \in G$ and $v = (1, 0) \in T_{(c_0,w_0)}G$ be a vertical tangent vector.
Then the vector $v$ is admissible and its Lorentzian length is equal to $\sqrt{1+|w_0|^2}$ and increases to the infinity while $|w_0| \rightarrow +\infty$.
\end{lemma}

\begin{proof}
Let us compute the pre-image of the vector $v$ with respect to the left-shift, i.e., $L_{(c_0,w_0) *}^{-1} v = (\xi, \omega) \in T_{(0,0)}G$.
With the help of Lemma~\ref{lem-admissible-velocities} we get
$$
\xi + \frac{\Image{w_0\bar{\omega}e^{-ic_0}}}{\sqrt{1+|w_0|^2}} = 1, \qquad \omega = \frac{iw_0\xi e^{-ic_0}}{\sqrt{1+|w_0|^2}}.
$$
Substituting $\omega$ from the second expression to the first one we obtain
$$
\xi - \frac{\Image{i|w_0|^2\xi}}{1+|w_0|^2} = \xi \left(1 - \frac{|w_0|^2}{1+|w_0|^2} \right) = \frac{\xi}{1+|w_0|^2} = 1.
$$
First, note that the vector $(\xi, \omega)$ is admissible. Indeed,
$$
\xi^2 - |\omega|^2 = (1+|w_0|^2)^2 - |w_0|^2(1+|w_0|^2) = 1+|w_0|^2 > 0.
$$
Thus, the vector $v$ is also admissible.
Second, the Lorentzian length of this vector equals $\sqrt{1+|w_0|^2}$.
\end{proof}

To prove item~(2) of Theorem~\ref{th-A} it remains to show that for any point of the set $\A_{exist}$ there exists the longest arc from the point $(0,0)$.
We will do it in Section~\ref{sec-cc-cut-locus} after studying the diffeomorphic properties of the Lorentzian exponential map, see Proposition~\ref{prop-symmetric-existence}.

\subsection{\label{sec-atset-prolate}The prolate case}

Consider now the case $\eta > 0$. First, we will prove the existence of the longest arcs on the attainable set (see Lemma~\ref{lem-prolate-exist} below). Second, this will allow us to describe the attainable set, since it coincides with the attainable set along extremal trajectories in this case (see Proposition~\ref{prop-atset-prolate}). To prove the existence we use the following result.

We will formulate this result in general situation for a regular (sub-)Lorentzian manifold $M$ equipped also with a Riemannian manifold structure (for instance, if $M=\R^{n+1}$, the standard Euclidean structure suffices). We will denote the lengths of vectors $v \in T_xM$ in this structure as $|v|_R$, and the distance between points $x, y \in M$ as $\dist_R(x, y)$.

\begin{theorem}[see Theorem~1 in~\cite{lokutsievskiy-podobryaev}]
\label{th-lp-exist}
Let $M$ be a regular \emph{(}sub-\emph{)}Lorentzian manifold, and $M$ is a complete Riemannian manifold. Let $x_0 \in M$ be any fixed point and $C^+_x \subset T_xM$ be the set of admissible velocities at a point $x \in M$. Suppose there exists a 1-form $\tau \in \Lambda^1(M)$ such that
\begin{enumerate}
	\item[\emph{(1)}] $\forall\,x\in M$ $\forall\,v \in C^+_x$ we have $\tau_x(v) \geqslant \frac{|v|_R}{1+\dist_R(x_0,x)}$, in particular, $\tau_x(v) > 0$,
	\item[\emph{(2)}] $d\tau = 0$.
\end{enumerate}
Assume that $H^1(M) = 0$. Then there exists the longest arc going from the point $x_0 \in M$ to a point $x_1 \in M$ if and only if there is at least one admissible path from $x_0$ to $x_1$.
\end{theorem}

\begin{lemma}
\label{lem-prolate-exist}
Assume that $\eta > 0$. Then any point of the attainable set can be reached by the longest arc from the identity point.
\end{lemma}

\begin{proof}
Let us apply Theorem~\ref{th-lp-exist}. Our manifold $G \backsimeq \R \times \C$ is simply connected and complete with respect to the standard Euclidian structure $\sqrt{c^2 + |w|^2}$, where $c \in \R$, $w \in \C$. Take $x_0 = (0,0)$ and 1-form $\tau = dc$ which is closed. It remains to verify condition~(1) of Theorem~\ref{th-lp-exist}, namely we need to check that $\frac{|v|_R}{\tau_{(c,w)}(v)} \leqslant 1 + \dist_R{((0,0), (c,w))}$.

From Lemma~\ref{lem-admissible-velocities} we know the admissible velocities $v$ at a point $(c,w) \in G$.
Note that $|\omega| \leqslant \frac{\xi}{\sqrt{\eta+1}}$, since $I_1\xi^2 - I_2|\omega|^2 \geqslant 0$ and $\frac{I_2}{I_1} = \eta + 1 > 0$, $\xi > 0$.
We have
$$
\tau_{(c,w)}(v) = \frac{\xi\sqrt{1+|w|^2} + \Image{w\bar{\omega}e^{-ic}}}{\sqrt{1+|w|^2}}.
$$
Hence, we can estimate
$$
0 < \frac{\xi\left(\sqrt{1+|w|^2} - \frac{|w|}{\sqrt{\eta+1}}\right)}{\sqrt{1+|w|^2}}
\leqslant \tau_{(c,w)}(v) \leqslant
\frac{\xi\left(\sqrt{1+|w|^2} + \frac{|w|}{\sqrt{\eta+1}}\right)}{\sqrt{1+|w|^2}}.
$$
So, $\tau_{(c,w)}(v)$ is positive for any $w$ and any $\omega$ such that $|\omega| \leqslant \frac{\xi}{\sqrt{\eta+1}}$ iff
$\sqrt{1+|w|^2} - \frac{|w|}{\sqrt{\eta+1}} > 0$ for any $w$. The last condition holds iff $\eta \geqslant 0$.

Moreover, $\frac{\sqrt{1+|w|^2} \pm \frac{|w|}{\sqrt{\eta+1}}}{\sqrt{1+|w|^2}} \rightarrow 1 \pm \frac{1}{\sqrt{\eta+1}} > 0$
while $|w| \rightarrow +\infty$ iff $\eta > 0$.
Thus, there exist some constants $C_1, C_2 > 0$ such that  $C_1 < \frac{\sqrt{1+|w|^2} \pm \frac{|w|}{\sqrt{\eta+1}}}{\sqrt{1+|w|^2}} < C_2$.
Hence,
\begin{equation}
\label{eq-estimation}
\frac{|v|_R}{\tau_{(c,w)}(v)} < \frac{\sqrt{C_2^2\xi^2 + |\omega\sqrt{1+|w|^2}e^{ic} - iw\xi|^2}}{C_1\xi} =
\frac{1}{C_1}\sqrt{C_2^2 + \left|\frac{\omega}{\xi}\sqrt{1+|w|^2}e^{ic} - iw\right|^2}.
\end{equation}
Using the following bound
$$
\left|\frac{\omega}{\xi}\sqrt{1+|w|^2}e^{ic} - iw\right| \leqslant \left|\frac{\omega}{\xi}\sqrt{1+|w|^2}e^{ic}\right| + |iw| \leqslant
$$
$$
\leqslant \frac{\sqrt{1+|w|^2}}{\sqrt{\eta+1}} + |w| \leqslant \frac{1+|w|}{\sqrt{\eta+1}} + |w|,
$$
we obtain
$$
\frac{|v|_R}{\tau_{(c,w)}(v)} < \frac{1}{C_1}\sqrt{C_2^2 + \left(\frac{1+|w|}{\sqrt{\eta+1}} + |w|\right)^2} =
\frac{C_2}{C_1}\sqrt{1 + \frac{1}{C_2^2}\left(\frac{1+|w|}{\sqrt{\eta+1}} + |w|\right)^2} \leqslant
$$
$$
\leqslant \frac{C_2}{C_1}\left(1 + \frac{1}{C_2}\left(\frac{1+|w|}{\sqrt{\eta+1}} + |w|\right)\right) = A + B|w|,
$$
for some constants $A, B > 0$. So, we get the following bound
$$
\frac{|v|_R}{A\tau_{(c,w)}(v)} = \frac{\frac{B}{A}|v|_R}{B\tau_{(c,w)}(v)} < 1 + \frac{B}{A} \dist_R{((0,0), (c,w))}.
$$
Whence, the required 1-form is $B\tau$ and the required Riemannian structure is $\frac{B}{A}\dist_R{(\cdot,\cdot)}$. It follows that the conditions of Theorem~\ref{th-lp-exist} are satisfied. This implies that the longest arc from the identity point to any point of the attainable set exists.
\end{proof}

\begin{proposition}
\label{prop-atset-prolate}
Assume that $\eta > 0$.\\
\emph{(1)} The attainable set reads as
$$
\A = \Bigl\{(c,w) \in G \, \Bigm| \, c \geqslant
\arctan{\Biggl(\frac{\tan{(\tau\sqrt{\eta})} + \frac{1}{\sqrt{\eta}}\tanh{\tau}}{1-\frac{1}{\sqrt{\eta}}\tanh{\tau}\tan{(\tau\sqrt{\eta})}}\Biggr)}, \ \text{where} \
\tau = \arcsinh{\Bigl(|w|\frac{\sqrt{\eta}}{\sqrt{\eta+1}}\Bigr)} \Bigr\}.
$$
\emph{(2)} Any point of $\A$ can be reached from the point $(0,0)$ by the longest arc.
\end{proposition}

\begin{figure}[h]
  \center{\includegraphics[width=0.32\linewidth]{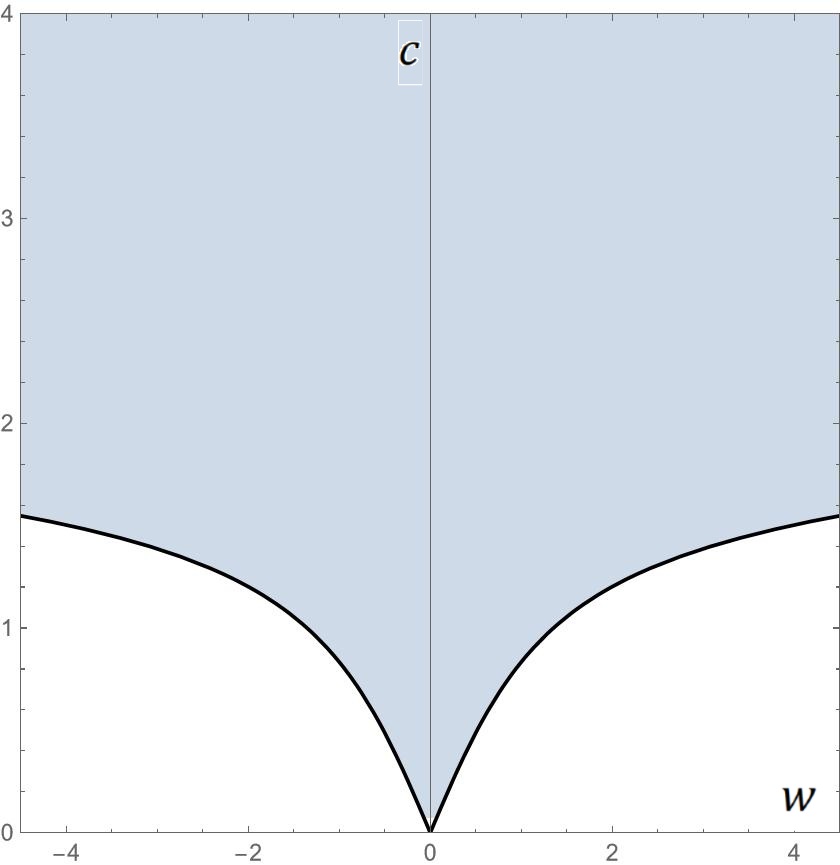}}
\caption{\label{pic-prolate-case-atset}The attainable set for the prolate case where $\eta = 0.1$ (the surface of revolution with respect to the $c$-axis).}
\end{figure}

\begin{proof}
See Fig.~\ref{pic-prolate-case-atset} for the attainable set.

(1) First, notice that the condition in the definition of the set $\A$ means that the attainable set is the set of points that are located upper than the light-like extremal trajectories. Indeed, the light-like extremal trajectories are defined by formulas~\eqref{eq-circle-space-geodesic}, where $\bar{h}_1 = -\frac{1}{\sqrt{\eta}}$ and
$\sqrt{\bar{h}_2^2 + \bar{h}_3^2} = \frac{\sqrt{\eta+1}}{\sqrt{\eta}}$, since for an initial covector of a light-like geodesic in the case $\eta > 0$ we have
$$
\bar{h}_1^2 - \bar{h}_2^2 - \bar{h}_3^2 = -1, \qquad \frac{h_1^2}{I_1} - \frac{h_2^2}{I_2} - \frac{h_3^2}{I_2} = 0.
$$
Second, any point that is located upper than the light-like geodesics belongs to the attainable set due to Lemma~\ref{lem-att-up}.
So, it remains to show that any point of a light-like geodesic is a point of the boundary $\partial\A$ of the attainable set.

Let us show that $\partial\A \neq \varnothing$. Indeed, in the opposite case by Krener's theorem (see, for example, \cite[Th.~8.1]{agrachev-sachkov}) $\A = G$.
This implies that we can attain the point $(-1,0)$, then by Lemma~\ref{lem-att-up} we can attain the point $(0,0)$.
We get a closed admissible path through the point $(0,0)$. Thus, there is no optimal solution for the point $(0,0)$ in contradiction with Lemma~\ref{lem-prolate-exist}.

Next, if $x \in \partial\A$, then according to the geometrical form of the Pontryagin maximum principle~\cite[Th.~12.1]{agrachev-sachkov} there exists an abnormal extremal trajectory coming to the point $x$. But due to Proposition~\ref{prop-energy}~(1) any abnormal geodesic is light-like. Hence, the set $\partial\A$ coincides with the set swaped by the light-like geodesics.

(2) immediately follows from Lemma~\ref{lem-prolate-exist}.
\end{proof}

Theorem~\ref{th-A}~(3) in the three dimensional case immediately follows from Proposition~\ref{prop-atset-prolate}.

\section{\label{sec-def}Optimality of geodesics}

In this section, we recall some basic definitions and the scheme for investigation of geodesics for optimality.

Any extremal is a trajectory of the Hamiltonian vector field $\vec{H}$ on the cotangent bundle $T^*G$. Hence, an extremal is defined by its initial point, i.e., a point (called \emph{an initial covector}) in the fiber $T^*_{\id}G \backsimeq \g^*$ of the cotangent bundle. This gives rise to the following definition.

\begin{definition}
\label{def-exp}
\emph{The exponential map} is the map
$$
\Exp : \g^* \times \R_+ \rightarrow G, \qquad \Exp(h, t) = \pi \circ e^{t\vec{H}} h, \qquad (h, t) \in \g^* \times \R_+,
$$
where $e^{t\vec{H}}$ is the flow of the Hamiltonian vector field $\vec{H}$ and $\pi : T^*G \rightarrow G$ is the canonical projection.
\end{definition}

So, a geodesic can be written as $\Gamma_{h, t_1} = \{\Exp(h, t) \, | \, t \in [0,t_1]\}$ for some $h \in \g^*$.

\begin{definition}
\label{def-cut}
A time $\tcut(h) \in \R_+ \cup \{+\infty\}$ is called \emph{the cut time for a geodesic with an initial covector} $h$,
if $\Gamma_{h, t_1}$ is optimal for $t_1 \leqslant \tcut(h)$ and is not optimal for $t_1 > \tcut(h)$.
The corresponding point $\Exp(h, \tcut(h))$ is called \emph{a cut point}.
The set of cut points for all geodesics starting at the identity point is called \emph{the cut locus} and denoted by $\Cut$.
\end{definition}

Our goal is to describe the cut locus and to find the cut time.
There are two reasons for a geodesic to lose optimality. The first one is existence of Maxwell points.

\begin{definition}
\label{def-maxwell}
A time $\tmax(h) \in \R_+ \cup \{+\infty\}$ is called \emph{a Maxwell time} if there exist two different geodesics with initial covectors $h$ and $\hat{h}$ meeting one another at the time $\tmax(h) = \tmax(\hat{h})$ with the same Lorentzian lengths. The corresponding point $\Exp(h, \tmax(h)) = \Exp(\hat{h}, \tmax(\hat{h}))$ is called \emph{a Maxwell point}.
\end{definition}

It is well known, that after a Maxwell point an extremal trajectory can not be optimal, see for example~\cite[Prop.~2.1]{sachkov-didona}.
That is why we are interested in the first Maxwell time.

The second reason is existence of conjugate points.

\begin{definition}
\label{def-conj}
A time $t$ is called \emph{a conjugate time along a geodesic with initial covector} $h \in \g^*$, if the point $(h,t)$ is a critical point of the exponential map.
The corresponding point $\Exp(h, t)$ is called \emph{a conjugate point}.
\end{definition}

After a conjugate point a geodesic loses local optimality. That is why we are interested in the first conjugate time $\tconj(h) \in \R_+ \cup \{+\infty\}$ and
the set $\Conj$ of the first conjugate points along all geodesics starting from the identity point (\emph{the first caustic}).

Usually Maxwell points appear for symmetric extremal trajectories. Let us give the formal definition.

\begin{definition}
\label{def-sym}
A pair of diffeomorphisms
$$
s : \g^* \times \R_+ \rightarrow \g^* \times \R_+, \qquad S : G \rightarrow G, \qquad \text{such that} \qquad \Exp \circ s = S \circ \Exp
$$
is called \emph{a symmetry of the exponential map}.
\end{definition}

\begin{proposition}
\label{prop-sym}
The group of symmetries of the exponential map for problem~\emph{\eqref{eq-optimal-control-problem}} contains the group $\SO_2$ that
acts by rotations around the axis $h_1$ in $\g^*$ and keeps the time in the pre-image of the exponential map and
acts by rotations around the axis $c$ in the image of the exponential map.
\end{proposition}

\begin{proof}
The dual map for any such a rotation is an automorphism of the Lie algebra $\g$. Thus, it induces a symmetry of the exponential map, see~\cite[Th.~1]{podobryaev-symmetries}.
\end{proof}

We will see below in Section~\ref{def-cut} that these symmetries are enough to obtain the cut locus. More precisely, following~\cite{sachkov-didona} we prove that the exponential map restricted to the open set bounded by the first Maxwell time corresponding to these symmetries is a diffeomorphism on to the open set $\A \setminus \M$,
where $\M$ is the set of corresponding first Maxwell points. Hence, for any point of the set $\A \setminus \M$ there is unique geodesic coming to this point.
If we know that the optimal solutions on the attainable set $\A$ exist, then this unique geodesic is optimal.
This will imply that $\M$ is the cut locus.

\section{\label{sec-conj-time}The first conjugate time}

Note that formulae~\eqref{eq-circle-time-geodesic}, \eqref{eq-circle-space-geodesic} for geodesics are quite similar up to some change of trigonometric functions to hyperbolic ones. Several necessary computations can be made in the same way. To avoid such a duplication let us introduce the following functions:
$$
c(\tau, h) = \left\{
\begin{array}{lll}
\cos{\tau}, & \text{if} & \Kil{h} < 0,\\
\cosh{\tau}, & \text{if} & \Kil{h} > 0,\\
\end{array}
\right.
\qquad
s(\tau, h) = \left\{
\begin{array}{lll}
\sin{\tau}, & \text{if} & \Kil{h} < 0,\\
\sinh{\tau}, & \text{if} & \Kil{h} > 0,\\
\end{array}
\right.
\qquad
\typeh = \sgn{\Kil{h}}.
$$
Note that the following equalities are satisfied:
$$
c^2(\tau, h) - \typeh s^2(\tau, h) = 1, \qquad c'(\tau, h) = \typeh s(\tau, h), \qquad s'(\tau, h) = c(\tau, h).
$$

\begin{theorem}
\label{th-cc-conj}
The first conjugate time has the following properties.\\
\emph{(1)} If $\eta < 0$, then $\tconj(h) \in (\frac{\pi I_2}{|h|}, \frac{2\pi I_2}{|h|}]$.\\
\emph{(2)} If $\eta = 0$, then
$$
\tconj(h) = \left\{
\begin{array}{lll}
\frac{2\pi I_2}{|h|}, & \text{if} & |h| \neq 0,\\
+\infty, & \text{if} & |h| = 0.\\
\end{array}
\right.
$$
\emph{(3)} If $\eta > 0$, then
$$
\tconj(h) = \left\{
\begin{array}{lll}
\frac{2\pi I_2}{|h|}, & \text{if} & \Kil{(h)} < 0,\\
+\infty, & \text{if} & \Kil{(h)} \geqslant 0.\\
\end{array}
\right.
$$
\end{theorem}

\begin{proof}
A point $(h, t) \in \g^* \times \R_+$ is a critical point of the exponential map iff
it is a critical point of the composition $\Pi \circ \Exp : \g^* \times \R_+ \rightarrow \SU_{1,1}$ of the exponential map and the projection $\Pi$ to the group $\SU_{1,1}$. This composition is defined by formulas~\eqref{eq-circle-time-geodesic}--\eqref{eq-circle-space-geodesic}.
Moreover, the Jacobi matrix of the map $\Pi \circ \Exp$ is equal to the product of two matrices:
$$
\frac{\partial{(q_0,q_1,q_2,q_3)}}{\partial{(\tau,\bar{h}_1,\bar{h}_2,\bar{h}_3)}} \cdot
\frac{\partial{(\tau,\bar{h}_1,\bar{h}_2,\bar{h}_3)}}{\partial{(t,h_1,h_2,h_3)}},
$$
where the second multiplier is non degenerate. So, we need to find zeros of the Jacobian $J$ of the first one.

Let us compute the partial derivatives:
$$
\begin{array}{lll}
\partial q^e_0 / \partial \tau & = & (\typeh - \eta\hone^2) s(\tau, h) \cos{(\tau\eta\hone)} - \hone(1+\eta) c(\tau, h) \sin{(\tau\eta\hone)},\\
\partial q^e_0 / \partial \hone & =  & -\tau \eta c(\tau, h)\sin{(\tau\eta\hone)} -s(\tau, h)\sin{(\tau\eta\hone)} -\tau\eta\hone s(\tau, h) \cos{(\tau\eta\hone)},\\
\partial q^e_1 / \partial \tau & = & (-\typeh+\eta\hone^2)s(\tau,h)\sin{(\tau\eta\hone)} - \hone(1+\eta)c(\tau, h)\cos{(\tau\eta\hone)},\\
\partial q^e_1 / \partial \hone & = & -\tau\eta c(\tau, h)\cos{(\tau\eta\hone)} -s(\tau, h)\cos{(\tau\eta\hone)} + \tau\eta\hone s(\tau, h)\sin{(\tau\eta\hone)}.\\
\end{array}
$$
It is easy to see that
$$
J =
s^2(\tau, h) \left(\frac{\partial q^e_0}{\partial \tau}\frac{\partial q^e_1}{\partial \hone} - \frac{\partial q^e_0}{\partial \hone}\frac{\partial q^e_1}{\partial \tau}\right) =
s^3(\tau, h) \Bigl(-\tau\eta(\typeh+\hone^2)c(\tau, h) - (\typeh - \eta\hone^2) s(\tau, h)\Bigr).
$$
Let us estimate the first positive root of the expression in the brackets. Assume that $s(\tau, h) \neq 0$. Then $c(\tau, h) \neq 0$, since otherwise $s(\tau, h) = 0$.
Dividing the expression in the brackets by $c(\tau, h)$ we get that the first positive root of $J$ coincides with the first positive root of the equation
\begin{equation}
\label{eq-conj}
\tau\eta\frac{\typeh + \hone^2}{\typeh - \eta\hone^2} = - \frac{s(\tau, h)}{c(\tau, h)}.
\end{equation}

If $\eta < 0$, then $\typeh = -1$ and $\hone \in \left(-\frac{1}{\sqrt{-\eta}}, -1\right]$.
Indeed, for an initial covector $h$ such that $H(h) = 0$ we have
$$
\hone^2 - \bar{h}_2^2 - \bar{h}_3^2 = 1, \quad \frac{\hone^2}{I_1} - \frac{\bar{h}_2^2 + \bar{h}_3^2}{I_2} = 0, \quad \hone < 0 \qquad \Rightarrow \qquad
\hone = -\frac{1}{\sqrt{-\eta}}.
$$
Since $-1 < \eta < 0$, it follows, that $-1 - \eta\hone^2 < 0$. Also we have $-1 + \hone^2 \geqslant 0$ and the coefficient of $\tau$ in equation~\eqref{eq-conj} is not negative.
The right side of this equation is $-\tan{\tau}$.
It means that the first positive root of equation~\eqref{eq-conj} is located on the interval $\left(\frac{\pi}{2}, \pi\right]$.

If $\eta = 0$, then obviously the first positive root of equation~\eqref{eq-conj} is $\pi$.

To prove statements~(1)--(2) it remains to note that in this case the first multiplier of $J$ is $s(\tau, h) = \sin{\tau}$ with the first positive root equal to $\pi$ and use the definition of $\tau$.

Now let us consider the case $\eta > 0$.

If $\typeh = -1$, then $-1 - \eta \hone^2 < 0$ and $-1 + \hone^2 \geqslant 0$. So, the coefficient of $\tau$ in equation~\eqref{eq-conj} is not positive.
Moreover, it is easy to see that this coefficient $\eta\frac{-1+\hone^2}{-1 - \eta\hone^2} > -1$.
This means that the first positive root of this equation is located on the interval $\left[\pi, \frac{3\pi}{2}\right)$. Note that in this case the first multiplier of $J$ that is $s(\tau, h) = \sin{\tau}$ has the first positive root $\pi$.

If $\typeh = 1$, then $1 + \hone^2 > 0$ and if $1 - \eta \hone^2 > 0$ we get that the coefficient of $\tau$ in equation~\eqref{eq-conj} is positive.
Since the right side of this equation is $-\tanh{\tau}$, then this equation has no positive roots.
If $1 - \eta \hone^2 < 0$, then since $\eta > -1$ it follows that the coefficient $\eta\frac{1+\hone^2}{1 - \eta\hone^2} < -1$.
Since $\frac{d}{d\tau}|_{\tau = 0}(-\tanh{\tau}) = -1$, equation~\eqref{eq-conj} has no positive roots.
The first multiplier of $J$ is $s(\tau, h) = \sinh{\tau}$ and it has no positive roots.

To complete the proof of statements~(2)--(3) it remains to show that if $\Kil{h} = 0$, then $\tconj(h) = +\infty$.
We will show that there are no conjugate points at a geodesic with such an initial covector.
By contradiction assume that there is a finite conjugate time $\tconj(h) < +\infty$.
Since the conjugate points on the geodesic are isolated~\cite{agrachev}, there exists $t_1 > \tconj(h)$ such that $t_1$ is not a conjugate time for this geodesic.
Consider a continuous curve $k : [0, 1] \rightarrow \g^*$ such that $k(0) = h$ and $\Kil{k(s)} < 0$ and $|k(s)| < \frac{2\pi I_2}{t_1}$ for $s \in (0, 1]$.
The number of conjugate points (taking into account multiplicity) on the geodesic arc
$\{q^s(t) = \Exp{(k(s), t)} \, | \, t \in [0, t_1]\}$ is equal to the Maslov index~\cite{arnold-index-maslova} of the path $l^s(t) = e^{-t \vec{H}}_* T^*_{q^s(t)}G$ in the Grassmanian of Lagrangian subspaces of $T_{(\id, 0)} T^*G$, see~\cite{agrachev}. Due to homotopic invariance of the Maslov index, the number of conjugate points on the geodesic arcs $\{q^0(t) \, | \, t \in [0, t_1]\}$ and $\{q^1(t) \, | \, t \in [0, t_1]\}$ are equal.
There are no conjugate points on the second arc, since the first conjugate time on this arc is $\frac{2\pi I_2}{|k(1)|} > t_1$. Thus, there are no conjugate points on the first one. We get a contradiction.
\end{proof}

\begin{remark}
\label{rem-conj-time-infinitely-small}
It follows from Theorem~\ref{th-cc-conj} that the conjugate time can be infinitely small in the case $\eta < 0$.
Indeed, note that for $h$ such that $H(h) = -\frac{1}{2}$ and $\Kil{h} < 0$ we have
$|h| = \frac{\sqrt{I_2}}{\sqrt{1 + \eta \hone^2}}$.
This easily follows from the substitution of $\bar{h}_2^2 + \bar{h}_3^2 = \hone^2 - 1$ to the condition:
$$
H(h) = -\frac{1}{2} |h|^2 \left( \frac{\hone^2}{I_1} - \frac{\bar{h}_2^2 + \bar{h}_3^2}{I_2} \right) = -\frac{1}{2} \qquad \Rightarrow \qquad
|h|^2 (\hone^2(1+\eta) - (\bar{h}_2^2 + \bar{h}_3^2)) = I_2.
$$
Moreover, in this case $\hone \in (-\frac{1}{\sqrt{-\eta}}, -1]$.
It follows $|h| \rightarrow +\infty$ while $\hone \rightarrow -\frac{1}{\sqrt{-\eta}}+0$.
Hence, due to Theorem~\ref{th-cc-conj}~(1) we obtain $\tconj(h) \rightarrow 0+0$.
\end{remark}

\section{\label{sec-maxwell-time}The first Maxwell time for symmetries}

Let us describe the first Maxwell points that correspond to the rotation symmetry.

\begin{proposition}
\label{prop-cc-maxwell}
\emph{(1)} If $\eta < 0$, then any point of the set $\M = \{(c, 0) \, | \, c \in (\pi(1-\sqrt{-\eta}), \pi(1+\eta)]\}$ is a Maxwell point,
where infinitely many time-like geodesics meet one another with the corresponding first Maxwell time $\tmax(h) = \frac{2 \pi I_2}{|h|}$.
These geodesics have initial covectors $h$ such that $\hone = \frac{\pi - c}{\pi\eta}$.\\
\emph{(2)} If $\eta = 0$, then the point $\M = \{(\pi, 0)\}$ is a Maxwell point where all of the time-like geodesics meet one another with the corresponding first Maxwell time $\tmax(h) = \frac{2 \pi I_2}{|h|}$.\\
\emph{(3)} If $\eta > 0$, then any point of the set $\M = \{(c, 0) \, | \, c \in [\pi(1+\eta), +\infty)\}$ is a Maxwell point,
where infinitely many time-like geodesics meet one another.
These geodesics have initial covectors $h$ such that $\hone = \frac{\pi - c}{\pi\eta}$.
The corresponding first Maxwell time equals
$$
\tmax(h) = \left\{
\begin{array}{lll}
\frac{2 \pi I_2}{|h|}, & \text{if} & \Kil{h} < 0,\\
+\infty, & \text{if} & \Kil{h} \geqslant 0.\\
\end{array}
\right.
$$
\end{proposition}

\begin{proof}
Immediately follows from formulae~\eqref{eq-circle-time-geodesic}. Indeed, the first time when $q_2 = q_3 = 0$ is $\tau = \pi$.
The corresponding value $c$ can be obtained as
$$
c = \arg{\{q_0^e(\pi) + iq_1^e(\pi)\}} = \arg{\{-\cos{(\pi\eta\hone)} + i\sin{(\pi\eta\hone)}\}} = \pi - \pi\eta\hone.
$$
It remains to use the definition of $\tau$ and the inclusions $\hone \in (-\frac{1}{\sqrt{-\eta}}, -1]$ for $\eta < 0$ and $\hone \in (-\infty, -1]$ for $\eta \geqslant 0$.
\end{proof}

\begin{remark}
\label{rem-cc-conj-maxwell}
Note that for $\eta \geqslant 0$ the first Maxwell time coincides with the first conjugate time and the Maxwell stratum coincides with the first caustic $\M = \Conj$.
But for $\eta < 0$ the first conjugate time is less than the first Maxwell time.
This phenomena does not appear in Riemannian geometry, where the first conjugate point appear earlier than the first Maxwell point only on geodesics that comes to the boundary of the cut locus.
\end{remark}

\section{\label{sec-cc-cut-locus}The cut locus}

In this section, we describe the cut locus for $\eta \geqslant 0$.
Note that in the case $\eta < 0$ the longest arcs do not exist by Proposition~\ref{prop-full-control}, thus the cut locus is trivial in this case.

Define the following sets.
First, the set of pair of initial covectors of time-like geodesics and times that are less than the first Maxwell time
$$
\mathcal{U} = \left\{ (h, t) \in \g^* \times \R_+ \, \Bigm| \, H(h) = -{{1}\over{2}}, \, 0 < t < \tmax(h)\right\},
$$
where $H$ is the Hamiltonian~\eqref{eq-H}.
Second, the subset $\V = \Exp{\mathcal{U}} \subset G$.
It follows from the results of Section~\ref{sec-atset-symmetric} that in the case $\eta = 0$
$$
\V = \{(c,w) \in G \, | \, \arctan{|w|} < c < \pi - \arctan{|w|}\},
$$
while in the case $\eta > 0$
$$
\V = \A \setminus (\partial\A \cup \M),
$$
since in this case the attainable set coincides with the attainable set along extremal trajectories, see~Proposition~\ref{prop-atset-prolate}.

\begin{proposition}
\label{prop-diffeomorphism}
The map $\Exp : \mathcal{U} \rightarrow \V$ is a diffeomorphism.
\end{proposition}

\begin{proof}
We will apply the Hadamard global diffeomorphism theorem~\cite{krantz-parks}, i.e., a smooth surjective nondegenerate proper map of two connected and simply connected manifolds of the same dimensions is a diffeomorphism.

Both manifolds $\mathcal{U}$ and $\V$ are connected, simply connected and three-dimensional. The map $\Exp: \mathcal{U} \rightarrow \V$ is smooth and surjective. Moreover, this map is non degenerate, since the first conjugate time is equal to the first Maxwell time in our case, see Theorem~\ref{th-cc-conj} and Proposition~\ref{prop-cc-maxwell}.
It remains to show that this map is proper, i.e., for any compact set $K \subset \V$ its pre-image $\Exp^{-1}{K} \subset \mathcal{U}$ is compact as well.

Assume by contradiction that $\Exp^{-1}{K}$ is not compact. Then there exist a sequence $(h^{(n)}, t^{(n)}) \in \Exp^{-1}{K}$ such that
at least one of the following conditions is satisfied (while $n \rightarrow +\infty$):
$$
\text{(a)} \ h^{(n)}_1 \rightarrow -\infty, \qquad \qquad
\text{(b)} \ t^{(n)} \rightarrow 0, \qquad \qquad
\text{(c)} \ t^{(n)} \rightarrow \tmax(h^{(n)}).
$$
Since $K$ is compact, then there exists a subsequence of the sequence $\Exp{(h^{(n)}, t^{(n)})}$ that converges to some element $q \in K$.
But in the case~(a) we have $q \in \partial\A$, in the case~(b) $q = (0,0)$. This contradicts with $K \subset \V$.

Consider now the case~(c). If there is a subsequence such that $\Kil{h^{(n)}} < 0$, then $q \in \M$ and we get a contradiction.
If there is no such a subsequence, then for $n$ big enough we obtain $\Kil{h^{(n)}} \geqslant 0$.
It follows that the sequence $\Exp{(h^{(n)}, t^{(n)})}$ is unbounded in contradiction with the compactness of the set $K$.
\end{proof}

\begin{proposition}
\label{prop-light-like-optimal}
In the case $\eta \geqslant 0$ the light-like geodesics are optimal to the infinity.
\end{proposition}

\begin{proof}
Assume by contradiction that there is an admissible curve coming to some point of a light-like geodesic with non-zero Lorentzian length.
Since light-like geodesics form the boundary of the attainable set $\A$ (see~Propositions~\ref{prop-symmetric-case-atset}, \ref{prop-atset-prolate}), then this curve is geometrically optimal and due to the Pontryagin maximum principle in geometric form~\cite[Th.~12.1]{agrachev-sachkov} this curve must be a light-like geodesic.
\end{proof}

\begin{proposition}
\label{prop-symmetric-existence}
In the symmetric case $I_1 = I_2 = I_3$ the longest arcs exist for the points of the set
$$
\A_{exist} = \{(c,w) \, | \, \atan{|w|} \leqslant c < \pi - \atan{|w|}\} \cup \{(\pi, 0)\}.
$$
\end{proposition}

\begin{proof}
See Fig.~\ref{pic-symmetric-case-atset}~(c) for the set $\A_{exist}$.

First, let us show that for any point of the set $\interior{\A_{exist}} = \{(c,w) \, | \, \atan{|w|} < c < \pi - \atan{|w|}\}$ the longest arc exists.
This is the consequence of the fact that any point of this set can be reached by the unique geodesic, see Proposition~\ref{prop-diffeomorphism}. A formal proof is based on the fields of extremals technic, see~\cite[Th.~17.2]{agrachev-sachkov}. Also an example of its application to the sub-Lorentzian structure on the Heisenberg group can be found in~\cite[Sec.~7.2]{sachkov-sachkova}.

Namely, consider the submanifold $L = \{e^{t\vec{H}}h \, | \, 0 < t < 2\pi I_2, \, h \in \g^*, H(h) = -\frac{1}{2} \} \subset T^*G$.
This submanifold is a Lagrange submanifold. Indeed, its dimension equals $\dim{G} = 3$ and the symplectic structure vanishes,
since it vanishes on the fiber of the cotangent bundle $T^*_{\id}G = \g^*$ and the Hamiltonian vector field $\vec{H}$ is symplectically orthogonal to the level set of the Hamiltonian $H$. Proposition~\ref{prop-diffeomorphism} implies that the Lagrange manifold $L$ is diffeomorphic to the domain $\interior{\A_{exist}}$. It follows that for any point of the set $\interior{\A_{exist}}$ the unique geodesic coming to this point is optimal.

Second, consider the lower part of the boundary of the set $\A_{exist}$. The light-like geodesics sweep this boundary $\{(c,w) \, | \, c = \atan{|w|}\}$.
By Proposition~\ref{prop-light-like-optimal} the light-like geodesics are optimal.

Third, let us prove that any geodesic from the one-parameter family of geodesics coming to the point $(\pi,0)$ is optimal.
We will use the semi-continuity of the Lorentzian distance, see~\cite[Lem.~4.4]{beem-ehrlich-easley}.
Namely, assume that $p,q \in M$ are two points of a Lorentzian manifold $M$ and there are two sequences of points $p_n \rightarrow p$, $q_n \rightarrow q$.
We will denote by $d(p,q)$ the Lorentzian distance from the point $p$ to the point $q$.
If $d(p,q) < +\infty$, then $d(p,q) \leqslant \liminf\limits_{n \rightarrow +\infty}{d(p_n,q_n)}$.
If $d(p,q) = +\infty$, then $d(p,q) \leqslant \lim\limits_{n \rightarrow +\infty}{d(p_n,q_n)} = +\infty$.
Let us take $p_n = p = (0,0)$ and $q_n \rightarrow q = (\pi,0)$, $q_n \in \interior{\A_{exist}}$.
It follows that $d(p,q) \leqslant \lim\limits_{n \rightarrow +\infty}{d(p,q_n)} = 2\pi I_2$.
But since all geodesics coming to the point $q$ have the Lorentzian length $2\pi I_2$, then $2\pi I_2 \leqslant d(p,q)$.
Hence, $d(p,q) = 2\pi I_2$ and any such geodesic is optimal.
\end{proof}

Now item~(2) of Theorem~\ref{th-A} follows from Propositions~\ref{prop-symmetric-case-atset} and \ref{prop-symmetric-existence} and
the proof of Theorem~\ref{th-A} in the three dimensional case is complete.

\begin{thm-hand}[\ref{th-B}.]
\label{th-cc-cut-locus}
Assume that $\eta \geqslant 0$. Then the cut locus coincides with the Maxwell set and the cut time equals the first Maxwell time.
Namely,
$$
\Cut = \left\{
\begin{array}{lll}
\{(\pi, 0)\}, & \text{if} & \eta = 0,\\
\{(c,0) \, | \, c \geqslant \pi(1+\eta)\}, & \text{if} & \eta > 0,\\
\end{array}
\right.
\qquad
\tcut(h) = \left\{
\begin{array}{lll}
\frac{2 \pi I_2}{|h|}, & \text{if} & \Kil{h} < 0,\\
+\infty, & \text{if} & \Kil{h} \geqslant 0.\\
\end{array}
\right.
$$
\end{thm-hand}

\begin{proof}
It follows from Propositions~\ref{prop-diffeomorphism}--\ref{prop-light-like-optimal} that there is unique geodesic coming to an arbitrary point of the set
$\A \setminus \M$. From Proposition~\ref{prop-symmetric-existence} and Proposition~\ref{prop-atset-prolate}~(2) we know that the longest arcs exist.
This implies that the set $\Cut = \M$ is in fact the cut locus.
Hence, we proved Theorem~\ref{th-B} in the three dimensional case.
\end{proof}

The Lorentzian geodesics are presented on Fig.~\ref{pic-geodesics}. The first Maxwell points are indicated there by thick points on the $c$-axis. One can see on Fig.~\ref{pic-geodesics}~(a) the envelope curve which corresponds to the first conjugate points. Notice that in the oblate case all of the geodesics have Maxwell points, see Fig.~\ref{pic-geodesics}~(a). In the symmetric case all of the time-like geodesics come to the Maxwell point $(\pi,0)$, while the light-like geodesics are optimal to the infinity (dashed lines), see Fig.~\ref{pic-geodesics}~(b). In the prolate case some of the time-like geodesics have Maxwell points (if $\Kil{h} < 0$, where $h$ is an initial covector), but the time-like geodesics with initial covectors $h$ such that $\Kil{h} \geqslant 0$ are optimal to the infinity as well as the light-like ones, see dashed lines on Fig.~\ref{pic-geodesics}~(c).

The wavefronts are plotted on Fig.~\ref{pic-wavefronts}. The singularities are located at the Maxwell set.

\begin{figure}[h]
\minipage{0.32\textwidth}
  \includegraphics[height=5cm,width=\linewidth]{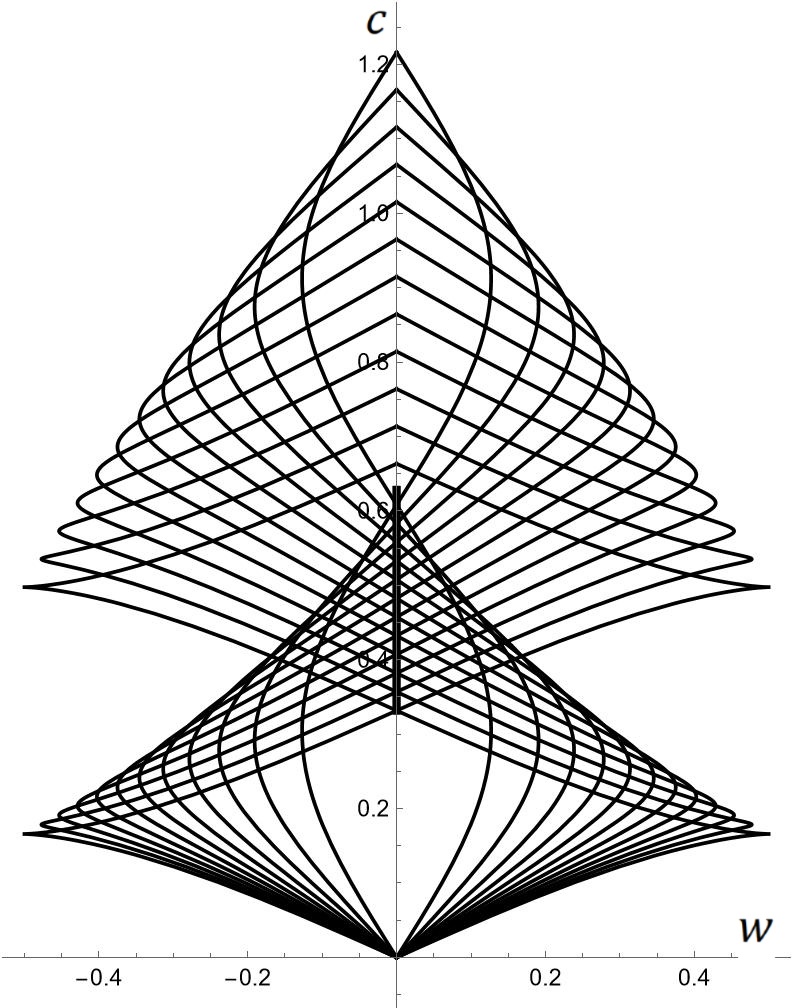}
  \\ \center{(a) The oblate case $\eta = -0.8$.}
\endminipage\hfill
\minipage{0.32\textwidth}
  \includegraphics[height=5cm,width=\linewidth]{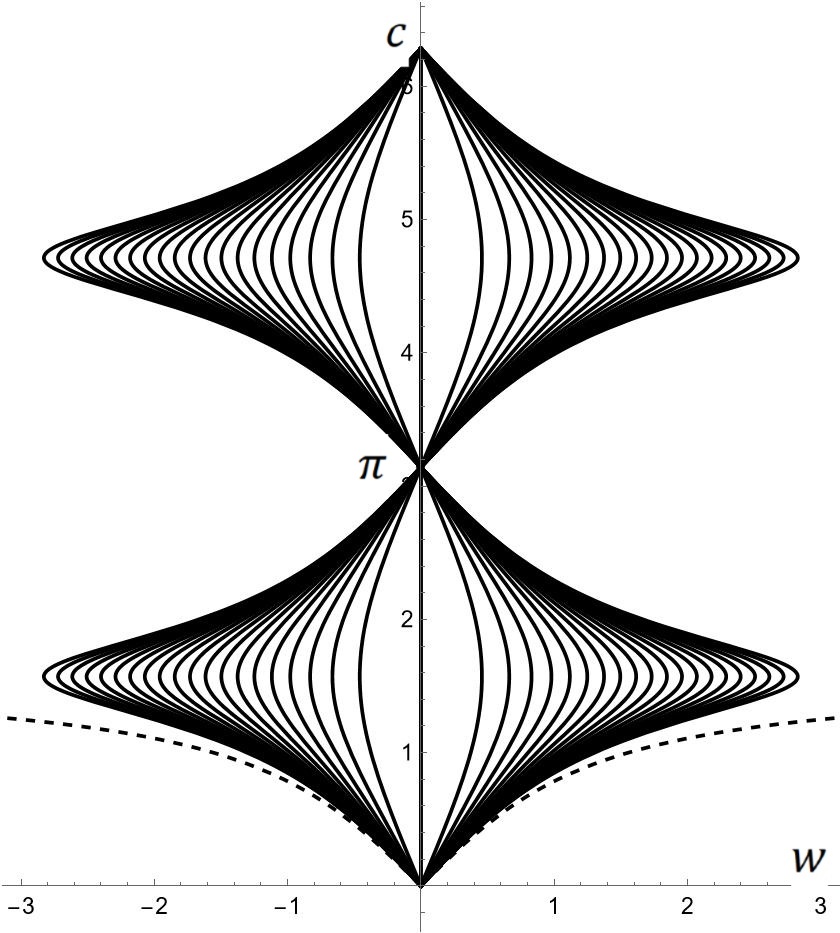}
  \\ \center{(b) The symmetric case $\eta = 0$.}
\endminipage\hfill
\minipage{0.32\textwidth}
  \includegraphics[height=5cm,width=\linewidth]{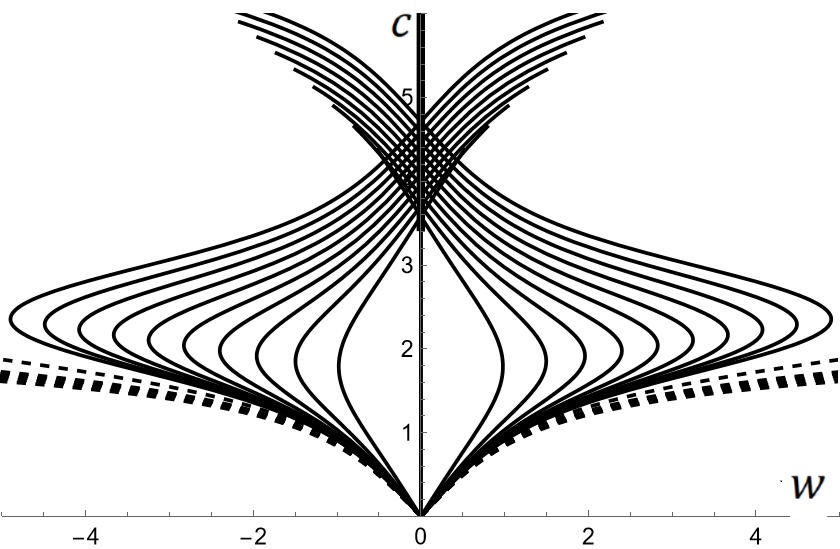}
  \\ \center{(c) The prolate case $\eta = 0.1$.}
\endminipage
\caption{\label{pic-geodesics}Geodesics that lose optimality (the solid lines), geodesics that are optimal to infinity (the dashed lines) and the Maxwell points (the thick points on the $c$-axis). The surface of revolution with respect to $c$-axis.}
\end{figure}

\begin{figure}[h]
\minipage{0.32\textwidth}
  \includegraphics[height=4cm,width=\linewidth]{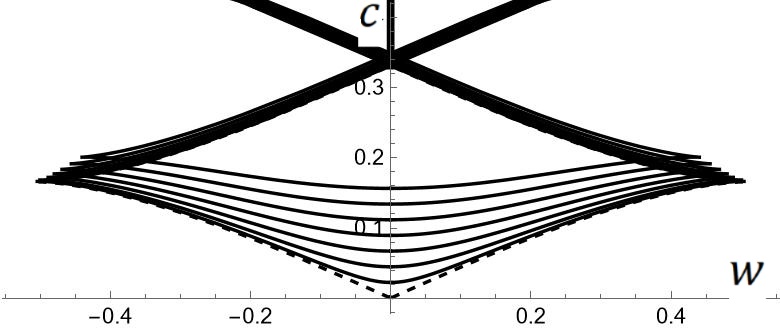}
  \\ \center{(a) The oblate case $\eta = -0.8$.}
\endminipage\hfill
\minipage{0.32\textwidth}
  \includegraphics[height=4cm,width=\linewidth]{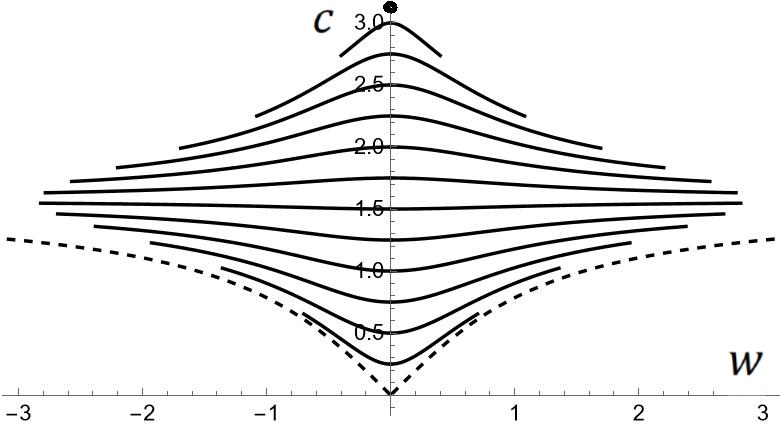}
  \\ \center{(b) The symmetric case $\eta = 0$.}
\endminipage\hfill
\minipage{0.32\textwidth}
  \includegraphics[height=4cm,width=\linewidth]{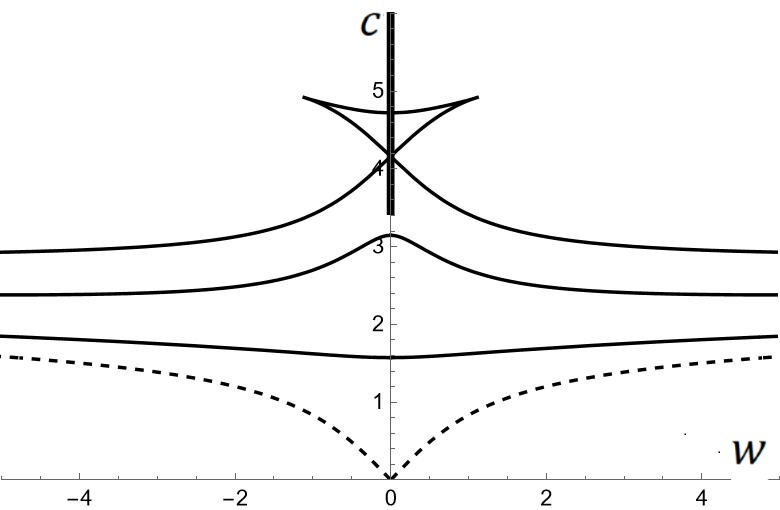}
  \\ \center{(c) The prolate case $\eta = 0.1$.}
\endminipage
\caption{\label{pic-wavefronts}The wavefronts (the solid lines) and the Maxwell points (the thick points on the $c$-axis). The dashed lines are the wavefront of zero radius (the light-like geodesics). The surface of revolution with respect to $c$-axis.}
\end{figure}

\section{\label{sec-inj-rad}The injectivity radius}

In this section, we compute the injectivity radius of our Lorentzian metric using the previous result on the cut time.

Let us recall the notion of the injectivity radius for Lorentzian manifolds. We follow~\cite[Def.~2.1]{chen-lefloch}.
Recall that we denote our Lorentzian quadratic form by $Q$. The Killing form $\Kil$ defines an isomorphism $\g \simeq \g^*$.
Thus, we can consider (with the help of this isomorphism) the quadratic forms $Q$ and $\Kil$ as quadratic forms on the space $\g^*$.
Fix some covector $p \in \g^*$ such that $H(p) = -\frac{1}{2}$ and $p_1 < 0$.
We will call this covector \emph{an observer momentum}, the corresponding velocity $(-p_1I_1, p_2I_2, p_3I_2)$ is the velocity of a time-like observer.
Consider the orthogonal complement to $p$ with respect to the bilinear form associated with $Q$ and denote it by $W = (\sspan{\{p\}})^{\perp_Q}$.
Note that the subspace $W$ is space-like.
Define a positive definite quadratic form $g_{R,p}$ as follows:
$$
g_{R,p}(h) = -Q(h_p) + Q(h_W), \qquad \text{where} \qquad h = h_p + h_W \in \g^*
$$
is a decomposition of an element $h$ with respect to the direct sum $\g^* = \sspan{\{p\}} \oplus W$.

\begin{definition}
The supremum of radii $r$ such that the restriction of the Lorentzian exponential map $\Exp$ to the Riemannian ball $B_r^p = \{h \in \g^* \, | \, g_{R,p}(h) < r^2\}$ is a diffeomorphism is called \emph{the injectivity radius $R_{inj}(p)$ with respect to an observer momentum $p$}.
\end{definition}

In other words, $R_{inj}(p) = \inf{\{\tcut(h) \, | \, h \in B_1^p\}}$.

\begin{thm-hand}[\ref{th-C}.]
\label{th-inj-rad}
The injectivity radius with respect to an observer momentum $p$ is equal to
$$
\Rinj(p) = \left\{
\begin{array}{lll}
0, & \text{if} & \eta < 0,\\
2\pi \frac{I_2}{\sqrt{|\lambda|}}, & \text{if} & \eta \geqslant 0,\\
\end{array}
\right.
$$
where $\lambda$ is the negative eigenvalue of the Killing form with respect to the Euclidian structure $g_{R,p}$.
\end{thm-hand}

\begin{proof}
First, consider the case $\eta < 0$. It follows from Remark~\ref{rem-conj-time-infinitely-small} that $\Rinj(p) = 0$.

Second, consider the case $\eta \geqslant 0$.
It follows from Theorem~\ref{th-B} that
\begin{equation}
\label{eq-sup}
R_{inj}(p) = 2\pi I_2 \Bigm/ \sup{\left\{|h| \, | \, h \in B_1^p, \, H(h) = -\frac{1}{2}, \, h_1 < 0, \, \Kil{h} < 0 \right\}}.
\end{equation}

Let us consider an orthonormal basis with respect to the positive definite bilinear form associated with $g_{R,p}$ that is orthogonal with respect to the Killing form $\Kil$.
The supremum in formula~\eqref{eq-sup} is taken over an open set that is time-like with respect to the Killing form $\Kil$.
Hence, this supremum is achieved on the element $x$ of this common basis such that the corresponding eigenvalue $\lambda = \Kil{x}$ of the form $\Kil$ is negative.
So, this supremum equals $\sqrt{|\lambda|}$.
Thus, the proof of Theorem~\ref{th-C} in the three dimensional case is complete.
\end{proof}

\begin{example}
\label{ex-injradius}
The simplest case is $p = (-\sqrt{I_1}, 0, 0)$, the corresponding velocity of the observer is directed along $c$-axis. We have already the required common diagonal forms for $g_{R,p}$ and $\Kil$. So, the supremum from the proof of Theorem~\ref{th-C} is achieved on the vector $x = p$.
Whence, $\lambda = \Kil{p} = -I_1$.
It follows that $R_{inj}(p) = 2\pi\frac{I_2}{\sqrt{I_1}}$.
\end{example}

\section{\label{sec-general}General case}

In this section, we consider the general case of the Lorentzian homogeneous space $\HH^{1,n} = \U_{1,n} / \U_n$ and its universal covering $M$ and reduce the corresponding problem to the three dimensional case $\HH^{1,1} = \U_{1,1} / \U_1$.

\begin{lemma}
\label{lem-totally-geodesic}
The submanifold $\HH^{1,1} = \{ (z,w) \in \HH^{1,n} \, | \, w_2 = \dots = w_n = 0 \} \subset \HH^{1,n}$ is totally geodesic, i.e.,
the geodesics of the submanifold $\HH^{1,1}$ are also geodesics of the manifold $\HH^{1,n}$.
\end{lemma}

\begin{proof}
The submanifold $\HH^{1,n-1} = \{ (z,w) \in \HH^{1,n} \, | \, w_n = 0 \} \subset \HH^{1,n}$ is totally geodesic, since
it is the set of the fixed points of the isometry $\diag{(1,\dots,1,e^{i\varphi})} \in \U_{1,n}$ where $\varphi \notin 2\pi \Z$.
It remains to apply this fact many times.
\end{proof}

\begin{lemma}
\label{lem-any-geodesic}
\emph{(1)} Any geodesic of the manifold $\HH^{1,n}$ starting from the point $o$ has the form $A g(t)$ for some $A \in \U_n$,
where $g(\cdot)$ is a geodesic of the manifold $\HH^{1,1}$ starting from the point $o$.\\
\emph{(2)} The same is true for the submanifold $\widetilde{\HH}^{1,1} \subset \widetilde{\HH}^{1,n}$ and $\widetilde{\U}_n$-action.
\end{lemma}

\begin{proof}
(1) The stabilizer $\U_n$ acts by isometries. Thus, it transforms geodesics to geodesics.
Moreover, it acts transitively on spheres in the subspace
$\{ (0,w) \, | \, w \in \C^n \} \subset \{ (ia,w) \, | \, a \in \R, \, w \in \C^n \} = T^*_o\HH^{1,n}$
of the initial covectors space.
Hence, choosing a suitable $A^{-1}$ we can move a geodesic initial covector to the subspace
$T^*_o\HH^{1,1} = \{(ia,w) \, | \, w_2 = \dots = w_n = 0 \} \subset T^*_o\HH^{1,n}$.
The corresponding geodesic is a geodesic of the submanifold $\HH^{1,1}$ due to Lemma~\ref{lem-totally-geodesic}.
The same arguments prove item~(2).
\end{proof}

\begin{lemma}
\label{lem-general-geodesic-atset}
The set attainable by geodesics for the general case has the same description as in the three dimensional case.
\end{lemma}

\begin{proof}
Immediately follows from Lemma~\ref{lem-any-geodesic},
Indeed, the conditions for the set attainable by geodesics in Theorem~\ref{th-A} include only $c$ and $|w|$ values.
But both of them are invariant under the $\widetilde{\U}_n$-action.
\end{proof}

\begin{proposition}
\label{prop-general-atset}
The attainable set for the general case has the same description as in the three dimensional case.
\end{proposition}

\begin{proof}
Obviously, the $\widetilde{\U}_n$-orbit of the three dimensional attainable set is attainable in the general case.

In the oblate case the three dimensional attainable set coincides with the whole space $\widetilde{\HH}^{1,1}$.
Thus, the attainable set in the general case coincides with the manifold $M = \widetilde{\HH}^{1,n}$.

Let us denote by $C^{1,n}_{(c,w)}$ the cone of the admissible velocities at a point $(c,w) \in M$ in the general case
By the same argument as in Lemma~\ref{lem-any-geodesic} we may assume that $(c,w) \in \widetilde{\HH}^{1,1}$ without loss of generality.
Let $C^{1,1}_{(c,w)}$ be the admissible velocities cone for the three dimensional problem on the manifold $\widetilde{\HH}^{1,1}$.

Consider the symmetric case. Let us prove that the light-like geodesics sweep the boundary of the attainable set.
We should show that the cone $C^{1,n}_{(c,w)}$ at any point $(c,w)$ of this surface is directed inside this set
like in the three dimensional case, see Section~\ref{sec-atset-symmetric}.
In the symmetric case, this is equivalent to the fact that the surface normal vector is located in the cone $C^{1,n}_{(c,w)}$.
Due to $\U_n$-symmetry the normal vector to this surface is the same as in Lemma~\ref{lem-light-like-boundary}.
Next, Lemma~\ref{lem-inside} shows that this normal vector is located in the cone $C^{1,1}_{(c,w)} \subset C^{1,n}_{(c,w)}$.

In the prolate case we can prove the existence of the longest arcs on the attainable set using Theorem~\ref{th-lp-exist}
the same way as in the three dimensional case, see Lemma~\ref{lem-prolate-exist}.
Namely, in fact this lemma gives the upper bound for the length of tangent vector such that their ends lie on the affine section $\tau_{(c,w)} = 1$
of the cone $C^{1,1}_{(c,w)}$.
Let us achieve the similar bound in the general case.
Remember that we assume without loss of generality that $(c,w) \in \widetilde{\HH}^{1,1}$.
It follows that the cone $C^{1,n}_{(c,w)}$ is symmetric with respect to the three dimensional subspace $\widetilde{\HH}^{1,1} \subset \R \times \C^n$ and
$$
C^{1,n}_{(c,w)} = \left\{
\left(
\xi + \frac{\Image{w^T\bar{\omega}e^{-ic}}}{\sqrt{1+|w|^2}}, \
A \omega - iw\xi
\right) \, \Bigm| \,
\xi > 0, \, \omega \in \C^n, \, |\omega| \leqslant \frac{\xi}{\sqrt{\eta+1}} \right\},
$$
where $A = \diag{(\sqrt{1+|w|^2}e^{ic},1,\dots,1)}$ is a diagonal matrix with $n-1$ unit elements.
So, for the bound~\eqref{eq-estimation} as in the proof of Lemma~\ref{lem-prolate-exist} we obtain
$$
\frac{|v|_R}{\tau_{(c,w)}(v)} < \frac{\sqrt{C_2^2\xi^2 + |A\omega - iw\xi|^2}}{C_1\xi}  =
\frac{1}{C_1}\sqrt{C_2^2 + \left|A\frac{\omega}{\xi} - iw\right|^2}.
$$
We have
$$
\left|A\frac{\omega}{\xi} - iw\right| \leqslant \left|A\frac{\omega}{\xi}\right| + |iw| \leqslant \frac{|\omega|}{\xi} \sqrt{n + |w|^2} + |w|
\leqslant \frac{\sqrt{n+|w|^2}}{\sqrt{\eta+1}} + |w| \leqslant \frac{n+|w|^2}{\sqrt{\eta+1}} + |w|,
$$
where the first inequality is the triangle inequality and the second one is the Cauchy-Bunyakovsky-Schwarz inequality.
Following the rest part of the proof of Lemma~\ref{lem-prolate-exist} we get the following bound:
$$
\frac{|v|_R}{\tau_{(c,w)}(v)} < \frac{C_2}{C_1}\left(1 + \frac{1}{C_2}\left(\frac{n+|w|}{\sqrt{\eta+1}} + |w|\right)\right) = A + B|w|,
$$
for some constants $A, B > 0$.

Thus, the longest arcs exist on the attainable set.
It remains to repeat the argument from the proof of Proposition~\ref{prop-atset-prolate}~(1), i.e., due to the existence of the longest arcs the attainable set has the boundary and this boundary consists of the light-like geodesics which have the same description as in the three dimensional case by Lemma~\ref{lem-general-geodesic-atset}.
\end{proof}

Thus, due to Lemma~\ref{lem-general-geodesic-atset} and Proposition~\ref{prop-general-atset} Theorem~\ref{th-A} is also true for the general case.

Let us now discuss the cut locus in the general case.
We need to show that if $(c,w) \in \widetilde{\HH}^{1,n}$ is an intersection point of two geodesics,
then the values $c$ and $|w|$ depend only on the Killing norm of the initial covectors of these geodesics (see the statement of Theorem~\ref{th-B}).

\begin{lemma}
\label{lem-general-intersection}
Assume that two geodesics $\gamma_1$ and $\gamma_2$ of the space $\HH^{1,n}$ intersect at a point $(z,w) \in \HH^{1,n}$.
Then there exist $g_1, g_2 \in \U_n$ such that the curves $g_1 \cdot \gamma_1$ and $g_2 \cdot \gamma_2$ are geodesics of the space $\HH^{1,1}$ and
intersect at a point $(z, w_1)$ where $|w_1| = |w|$.
\end{lemma}

\begin{proof}
It follows from Lemma~\ref{lem-any-geodesic} that there exist $g_1, g_2' \in \U_n$ such that
the curves $g_1 \cdot \gamma_1$ and $g_2' \cdot \gamma_2$ are geodesics of the space $\HH^{1,1}$.
Since $\U_n$-action keeps $z$-coordinate we get $g_1(z,w) = (z,w_1) \in \HH^{1,1}$ and $g_2'(z,w) = (z,w_2) \in \HH^{1,1}$. Moreover, $|w_1| = |w_2| = |w|$.
Hence, there exists an element $s \in \U_1$ keeping $\HH^{1,1}$ such that $s \cdot (z,w_2) = (z,w_1)$.
So, we can put $g_2 = sg_2'$.
\end{proof}

Since the $\U_n$-action keeps the Killing norm we get from Lemma~\ref{lem-general-intersection} that the statement of Theorem~\ref{th-B} is also true for the general case.

Finally, Theorem~\ref{th-C} in the general case follows from the fact that the Killing form $\Kil$, the Lorentzian form $Q$ and the Riemannian form $g_{R,p}$
(see the corresponding definition at the beginning of Section~\ref{sec-inj-rad}) are $\U_n$-invariant.

\begin{small}

\end{small}
\end{document}